\documentclass[11pt,reqno]{amsart}
\usepackage{diagrams}
\usepackage{amssymb}
\usepackage{cite}
\usepackage[pagewise]{lineno}
\numberwithin{equation}{section}

\theoremstyle{definition}
\newtheorem{thm}{Theorem}[section]
\newtheorem{cor}[thm]{Corollary}
\newtheorem{lem}[thm]{Lemma}
\newtheorem{exa}[thm]{Example}
\newtheorem{prop}[thm]{Proposition}
\newtheorem{defi}[thm]{Definition}
\newtheorem{rem}[thm]{Remark}
\newtheorem{note}[thm]{Notation}

\newtheorem*{shc}{Singular Hodge conjecture}
\newtheorem*{conj}{Conjecture A}

\newtheorem{notac}[thm]{Notations and Conventions}

\DeclareMathOperator{\red}{\mathrm{red}}

\DeclareMathOperator{\Ima}{\mathrm{Im}}

\DeclareMathOperator{\p3}{\mathbb{P}^3}

\DeclareMathOperator{\pr}{\mathrm{pr}}
\DeclareMathOperator{\go}{\mathrm{GO}}
\DeclareMathOperator{\Hom}{\mathrm{Hom}}

\DeclareMathOperator{\MT}{\mathrm{MT}}

\DeclareMathOperator{\GSp}{\mathrm{GSp}}

\newcommand{\mbb}[1]{\mathbb{#1}}
\newcommand{\mb}[1]{\mathbb{#1}}
\newcommand{\mc}[1]{\mathcal{#1}}
\newcommand{\mr}[1]{\mathrm{#1}}
\newcommand{\ov}[1]{\overline{#1}}
\newcommand{\wt}[1]{\widetilde{#1}}
\newcommand{\mf}[1]{\mathfrak{#1}}
\newcommand{\mbf}[1]{\mathbf{#1}}
\newcommand{\un}[1]{\underline{#1}}

\begin{document}

\title{Mumford Tate groups and the Hodge conjecture}

\author[A. Dan]{Ananyo Dan}

\thanks{A.D. is funded by EPSRC grant number  EP/T019379/1. I. K. was funded by the DFG, TRR $326$
Geometry and Arithmetic of Uniformized Structures, project number $444845124$ and is currently funded by 
EPSRC grant number EP/W026554/1}

\address{School of Mathematics and Statistics, University of Sheffield, Hicks building, Hounsfield Road, S3 7RH, UK}

\email{a.dan@sheffield.ac.uk}

\author[I. Kaur]{Inder Kaur}

\address{Department of Mathematical Sciences, Loughborough University, LE11 3TU, U.K}

\email{i.kaur@lboro.ac.uk}

\subjclass[2010]{$14$C$15$, $14$C$30$, $32$S$35$, $32$G$20$, $14$D$07$, $14$C$05$}
%$14$C$30$, $14$C$34$, , , $32$S$35$, $14$D$20$, Secondary $14$H$40$}

\keywords{Hodge conjecture,  Limit mixed Hodge
structures, Operational Chow group, Cycle class map, flag 
Hilbert schemes, singular varieties}

\date{\today}

\begin{abstract}
 In this article we study the (cohomological) Hodge conjecture for \emph{singular} varieties. We prove the 
 conjecture for simple normal crossing varieties that can be embedded in a family where the 
 Mumford-Tate group remains constant. We show how to produce such families. Furthermore, we show for 
 varieties with worse singularities the conjecture can be expressed solely in terms of the 
 algebraic classes.
\end{abstract}

\maketitle

\section{Introduction}
The underlying field will always be $\mbb{C}$.
Recall, the classical \emph{Hodge conjecture} claims that 
given a smooth projective variety $X$, every
(rational) Hodge class in $X$ is the cohomology class of an 
algebraic cycle in $X$. 
The conjecture is known in some cases (see \cite{voiconj, lewis} for a survey of known results and \cite{voihod, ink} for related
results), but is open in general.
A typical strategy has been to consider smooth, projective low dimensional varieties 
that are birational to already known cases. This is primarily because 
the exceptional divisors arising from the resolution of the indeterminacy locus satisfy the Hodge conjecture.
However, this strategy fails in higher dimension. Another approach is to consider families of 
varieties (e.g. in the case of abelian varieties) and then use a Noether-Lefschetz-type argument 
to conclude that the Hodge classes in a very general fiber in the family are powers of the first 
Chern class of a line bundle. This implies the Hodge conjecture for a very general fiber. 
In this article, we combine ideas from both these approaches.

 It is well-known that any smooth projective variety $X$ is birational to a hypersurface $X_{\mr{hyp}}$ in a projective space.
 This hypersurface $X_{\mr{hyp}}$ is almost always singular.
 Note that there is homological version of the Hodge conjecture for singular varieties given by Jannsen \cite[Conjecture $7.2$]{jann} (see also \cite{lewsing1}). 
 He proved that the classical Hodge conjecture is equivalent to the singular version 
 (see \cite[Theorem $7.9$]{jann}, see also \cite{lewsing2}).
 Therefore, proving the singular Hodge conjecture for $X_{\mr{hyp}}$ would imply the Hodge conjecture for $X$. 
  %In  Lewis uses this formulation to give a singular version of the generalized Hodge conjecture (again using homology) and in  proves the conjecture for a certain class of singular varieties.
%  However, there are no known cases of the singular Hodge conjecture, except the obvious ones.

In the present article, we give a cohomological formulation of the Hodge conjecture for singular varieties. There are obvious reasons why this interpretation has so far been unexplored. Firstly for $X$ singular, the classical Chow group is not compatible with pull-back morphisms. In \cite[Chapter $17$]{fult} (see also \cite[Proposition $4$]{soug}), Fulton and MacPherson developed the \emph{operational Chow group}, denoted $A^p(X)$ which is compatible with pull-back morphisms and for smooth varieties coincides with the classical Chow group. However, even for the operational Chow group, we know by \cite{totchow} that in general, there is no map $A^p(X) \to H^{2p}(X,\mb{Q})$ with good properties. Nevertheless, by the work of Bloch-Gillet-Soul\'{e} (see \cite{soub}) there is a (functorial) cycle class map:
\[\mr{cl}_p: A^p(X) \otimes \mathbb{Q} \to \mr{Gr}^W_{2p} H^{2p}(X,\mb{Q}).\]
Using this we formulate the cohomological singular Hodge conjecture as follows:

\begin{shc}
 Let $X$ be a projective variety such that the dimension of the singular locus is  at most $p-1$.
 Then, the image of the cycle class map
$\mr{cl}_p$ coincides with 
\[H^{2p}_{\mr{Hdg}}(X):= \mr{Gr}^W_{2p} H^{2p}(X,\mb{Q})
\cap F^p \mr{Gr}^W_{2p} H^{2p}(X,\mb{C}).\]  
\end{shc}

If $X$ is of dimension $n$ and the above conjecture holds for $X$, then we say that $X$ \emph{satisfies} $\mr{SHC}(p,n)$.
Of course, if $X$ is non-singular then the singular  
Hodge conjecture is the same as the classical Hodge conjecture. 
In this case, we say that $X$ \emph{satisfies} $\mr{HC}(p,n)$. 
 The Lefschetz $(1,1)$-theorem implies $\mr{HC}(1,n)$ holds true, for any $n$. 

%was formulated by Jannsen \cite[Conjecture $7.2$]{jann}
%in terms of Borel-Moore homology: for $X$ singular,
%projective of dimension $n$, the cycle class map 
%\[\mr{cl}_p: A_{n-p}(X) \to W_{2p-2n} H^{BM}_{2n-2p}(X,\mb{Q})(p-n)\cap F^{p-n}H^{\mr{BM}}_{2n-2p}(X, \mb{C})\]
%is surjective for any $p \ge 1$, 
%where $A^p(X) \otimes \mathbb{Q}$ is the free abelian group generated by codimension $p$ irreducible subvarieties of $X$, upto 
%rational equivalence and $\mr{BM}$ denotes the Borel-Moore homology.

Recall, a very general hypersurface of any dimension satisfies the Hodge conjecture (as the cohomology ring
is generated by the class of the hyperplane section).
 Therefore we can \emph{always} embed $X_{\mr{hyp}}$ in a one parameter family of hypersurfaces such that 
 a general fibre satisfies the Hodge conjecture. One then expects that the Hodge classes on $X_{\mr{hyp}}$ 
 ``spread out'' to Hodge classes in the family. Since a general member of the family satisfies the Hodge conjecture, we know that the Hodge class away from the centre is the cohomology class of an algebraic cycle. By the simple operation of taking closure, one can then extend the 
algebraic cycles on the general fiber to the central fiber. One needs to check that the cohomology class of this ``new'' algebraic cycle on the central fiber coincides with the Hodge class we started with. However, 
there are several technical problems. Heuristically, the specialization map is not injective and hence Hodge classes need not ``spread out''. Even if a Hodge class does spread out, it might not restrict to a Hodge class on the general fibre!
  In this article we study these problems and give several examples of families of  varieties where these problems can be circumvented. Let us make this precise. 

Let $X$ be a singular, projective variety of dimension $n$ and $\pi: \mc{X} \to \Delta$ be 
a flat family of projective varieties, smooth over $\Delta^*$  with the central fiber $X$. 
Fix an integer $p$. Denote by $\mf{h}$ the  universal cover for $\Delta^*$ and by $\mc{X}_\infty$ the 
pull-back of $\mc{X}$ to $\mf{h}$. By Ehresmann's theorem, for every $u \in \mf{h}$ there is an isomorphism of 
cohomology groups $H^{2p}(\mc{X}_\infty, \mb{Q})$ and $H^{2p}(\mc{X}_u,\mb{Q})$. The natural Hodge filtration on 
$H^{2p}(\mc{X}_u,\mb{Q})$ induces a filtration $F^p_u$ on $H^{2p}(\mc{X}_\infty,\mb{Q})$. The \emph{limit Hodge filtration} on 
$H^{2p}(\mc{X}_\infty,\mb{Q})$ arises as the limit of this filtration as the imaginary part of $u$ tends to $\infty$ (see \S \ref{sec:lhf} for details). However, there may be rational points  
$H^{2p}(\mc{X}_\infty,\mb{Q}) \cap F^pH^{2p}(\mc{X}_\infty, \mb{C})$ of the limit Hodge filtration
that \emph{do not} come from the rational points of the filtration $F^p_u$. The Noether-Lefschetz locus gives examples of this phenomena even for smooth families (see Example \ref{nonmtf}). 
As a result, $H^{2p}(\mc{X}_\infty,\mb{Q})$ may contain more Hodge classes than that on a general fiber!
This means that although a Hodge class on $\mc{X}_0$ maps to a Hodge class on $\mc{X}_\infty$ via the specialization map,
it need not extend to a Hodge class on the family.

The jump in the rank of the Hodge lattice is captured by Mumford-Tate groups (see \S \ref{sec:mtg} for the definition). 
We call $\pi$ a \emph{Mumford-Tate family} if the rank of the Mumford-Tate group remains ``constant in the limit'' (see \S \ref{sec:mtf} for precise definitions). Moreover, we call a singular, projective variety \emph{MT-smoothable} if it can be 
embedded as a central fiber of a Mumford-Tate family where the general fiber satisfies the Hodge conjecture.
We prove the following:

% Denote by $e: \mf{h} \to \Delta^*$ the universal covering and by $\mc{X}_\infty$ the pull-back of
% $\mc{X}$ to $\mf{h}$. By Ehresmann's theorem, given any $s \in \mf{h}$,
% there is a canonical identification between $H^{2p}(\mc{X}_\infty,\mb{Z})$ and $H^{2p}(\mc{X}_s, \mb{Z})$. 

\begin{thm}\label{thmintro1}
 Let $X$ be a projective variety of dimension $4$ with strict normal crossings singularities. 
 If $X$ is MT-smoothable, then $X$ satisfies $\mr{SHC}(p,4)$ for every $p$.
\end{thm}

%See ? for a proof. An obvious example satisfying the hypothesis of Theorem \ref{introthm1} is summarized by the following: 

 %\begin{cor}[Corollary \ref{cor02}]
  %Let $X$ be a hypersurface in $\mb{P}^{5}$ of 
  %degree at most $5$ and $\dim(X_{\mr{sing}})$ is at most $1$. 
  %If $X$ is the central fiber of a 
  %flat family of hypersurfaces in $\mb{P}^5$,
  %$\pi: \mc{Y} \to \Delta$ which is 
  %smooth over $\Delta^*$ and Mumford-Tate of weight $2$, 
  %then $X$ satisfies the singular Hodge conjecture.
  %In particular, the resolution of singularities of $X$ 
  %satisfies the Hodge conjecture.
 %\end{cor}

 In Theorem \ref{th01} below, we prove Theorem \ref{thmintro1} for any dimension. 
 Clearly Theorem \ref{thmintro1} leads to the following questions:
 
 \begin{itemize}
  \item {Question 1:} How to find Mumford-Tate families?
  \item {Question 2:} Can we generalize Theorem \ref{thmintro1} to varieties with worse singularities?
 \end{itemize}

For an exhaustive answer of Question $1$ one would need a complete description of the 
Noether-Lefschetz locus for families of hypersurfaces in all dimensions greater than $3$. This problem is largely open.
However in \S \ref{sec:exa}, 
 we give a general method to obtain Mumford-Tate families from known ones using the theory of correspondences. 
 Recall, that given a coherent sheaf $\mc{E}$ on a product of two smooth, projective varieties $X \times Y$, the $i$-th Chern class of $\mc{E}$ induces a morphism of pure Hodge structures 
 from $H^{2m-k}(X)$ to $H^{2i-k}(Y)$ for all integers $i$ and $k$, where $m=\dim(X)$ (see \S \ref{sec:cohgen}). 
 Let us denote such a morphism by $\Phi^{(i,k)}_{\mc{E}}$. 
 We say $Y$ is \emph{cohomologically generated by} $(X, \mc{E})$ if the cohomology ring $H^*(Y)$ is generated (as a ring) by the images of morphisms of the form $\Phi^{(i,k)}_{\mc{E}}$ as $i$ and $k$ varies over all integers (see Definition \ref{defi:coh}). Note that several examples of cohomologically generated varieties appear in existing literature. For example, in \cite{mumn} Mumford and Newstead proved that the moduli space of stable rank $2$ bundles with odd degree determinant over a curve $C$ is cohomologically generated by the pair $(C,\mc{U})$, where $\mc{U}$ is the universal bundle associated to the moduli space.
 In \cite{mark1, mark2} Markmann showed a similar result for moduli spaces of sheaves over certain surfaces. In \S 
 \ref{sec:exa} we show how this notion of cohomologically generated leads to producing more Mumford-Tate families.

\begin{thm}\label{introthm3}
    Let $\pi_1: \mc{X}^{*} \to \Delta^{*}$ and $\pi_2: \mc{Y}^{*}  \to \Delta^{*}$ be two smooth, projective families. Assume that there exists a coherent sheaf $\mc{U}$ over $\mc{X}^{*}\times_{\Delta^{*}} \mc{Y}^{*}$ such that it is flat over $\Delta^{*}$. Suppose that for general $t \in \Delta^*$,
    $\mc{Y}_t$ is cohomologically generated by $(\mc{X}_t, \mc{U}_t)$, where $\mc{U}_t:= \mc{U}|_{\mc{X}_t \times \mc{Y}_t}$. If the family $\pi_1$ is (strictly) Mumford-Tate family, then so is the family $\pi_2$. 
 \end{thm}

 See Theorem \ref{thm:mt} for the precise formulation.
An obvious choice for $\pi_1$ is a family of smooth curves degenerating to a singular curve (with arbitrary
singularities). See Proposition \ref{prop:cur} for a proof in the case when the singular curve is nodal.

Let us turn to Question $2$.  Suppose $X$ is a singular projective variety of dimension $n$ 
and $p$ be an integer such that $\dim(X_{\mr{sing}}) \le p-1$. Suppose $\phi: \wt{X} \to X$ is any resolution of 
 singularities and $E$ is the exceptional divisor. 
 By \cite[Corollary-Definition $5.37$]{pet}, we have an exact sequence on cohomology
 \[ H^{2p}(X) \to H^{2p}(\wt{X}) \to H^{2p}(E).\]
  We conjecture that taking algebraic cohomology groups preserves the exactness of the sequence:

  \begin{conj}\label{introconj}
   The following sequence is exact:
 \[ H^{2p}_A(X) \to  H^{2p}_A(\wt{X}) \to H^{2p}_A(E).\]
\end{conj}

Note that, this conjecture does not involve Hodge classes. Surprisingly, we prove that 
if $X$ is MT-smoothable, then this conjecture is equivalent to the singular Hodge conjecture. In particular, 

\begin{thm}\label{introthm2}
Let $X$ be as above.
If $X$ satisfies $\mr{SHC}(p,n)$, then $X$ satisfies Conjecture A. 
Conversely, if $\mr{HC}(p-1,n-1)$ holds true, 
$X$ is MT-smoothable  and satisfies Conjecture A, then $X$ 
satisfies $\mr{SHC}(p,n)$.
\end{thm}

See Theorem \ref{th04} for the precise statement.

 $\bf{Outline}$: The paper is organised as follows: in \S \ref{sec:prelim}
 we briefly recall the necessary preliminaries on limit mixed Hodge structures and flag Hilbert schemes. 
 In \S \ref{sec:mt} we recall the definition of a Mumford-Tate group and introduce Mumford-Tate families. 
 We give both examples and non-examples of such families. In \S \ref{sec:shc}, we define limit algebraic cohomology 
 groups and limit Hodge classes. We recall the preliminaries on Operational Chow group and the Bloch-Gillet-Soul\'{e} cycle class map. We give the singular Hodge conjecture and prove some of the preliminary results which we use later. In \S \ref{sec:main}, we prove the main results of this article. Finally, in \S \ref{sec:exa} we give a method to produce Mumford-Tate families.

% 
% show that $H^{2p}_{\mr{Hdg}}(\mc{X}_\infty)$
% is  generated by algebraic classes. Using limit mixed Hodge structures ensures that the specialization morphism 
% \[ \mr{sp}_{\pi_0}: H^{2p}_{\mr{Hdg}}(X) \to H^{2p}_{\mr{Hdg}}(\mc{X}_\infty) \]
% is indeed a morphism of mixed Hodge structures. However, it is a priori neither injctive nor surjective. Indeed, there are several examples of families such that   $ker(\mr{sp}_{\pi_0})$ is non-trivial. In particular, $ker(\mr{sp}_{\pi_0})$ contains the Hodge classes on the singular fibre $X$ that do not extend to the whole family $\mc{X}_\infty$. To prove Theorem \ref{mainthmintro} we need to understand this.
%In \cite{schvar}, Schmid showed that there exists a (specialization) morphism of mixed Hodge structures  from $H^{2p}_{\mr{Hdg}}(\mc{X}_0,\mb{C})$ to $H^{2p}_{\mr{Hdg}}(\mc{X}_{\infty},\mb{C})$, for $s \in \Delta^{*}$. 

 \section{Preliminaries}\label{sec:prelim}
 In this section we briefly recall some of the basics on limit mixed Hodge  structures and 
 flag Hilbert schemes. 
 Limit mixed Hodge structures play an important role throughout this article. 
 See \cite[\S $11$]{pet} for a detailed treatment  of the topic.
 
  \subsection{Setup}\label{se01}
Consider a flat family of projective varieties,
 \[\pi:\mc{X} \to \Delta,\]  smooth over $\Delta^*$ of relative dimension $n$. Suppose 
  the central fiber $\mc{X}_0:= \pi^{-1}(0)$ 
 is a reduced, simple normal crossings divisor.
 Denote by $ \pi':\mc{X}_{\Delta^*} \to \Delta^*$
 the restriction of $\pi$ to the punctured disc $\Delta^*$.
  Denote by $X_1,...,X_r$
 the irreducible components of the central fiber $\mc{X}_0$.
For $m \ge 2$, denote by $X(m)$ the disjoint union of 
the intersections of $m$ number of irreducible components of 
$\mc{X}_0$ i.e., 
\[X(m) := \coprod_{\substack{|I|=m\\ I=(1 \le i_1<i_2<...<i_m \le r)}}
 \left(\bigcap\limits_{k=1}^m X_{i_k}\right).\]

 Let $e: \mf{h} \to \Delta^*$ be the exponential map 
from the upper half plane $\mf{h}$ to the punctured disc 
$\Delta^*$. Denote by $\mc{X}_\infty:=\mc{X}_{\Delta^*} \times_{\Delta^*} \mf{h}$ the base change of $\mc{X}_{\Delta^*}$
to $\mf{h}$ via the exponential map $e$.

\subsection{Monodromy operator}\label{sec:mon}
Since $\mf{h}$ is simply connected, 
the natural inclusion 
\[i_s: \mc{X}_{e(s)} \hookrightarrow \mc{X}_\infty\]  
for any $s \in \mf{h}$, induces an isomorphism of 
cohomology groups:
\[i_s^*: H^{2p}(\mc{X}_\infty, \mb{Z}) \xrightarrow{\sim}
 H^{2p}(\mc{X}_{e(s)}, \mb{Z}).\]
 Note that, the morphism $i_s^*$ changes even if $e(s)$ does not. In particular, we have the 
 \emph{monodromy operator} associate to the family $\pi$ given by the composition:
 \[T: H^{2p}(\mc{X}_\infty, \mb{Z}) \xrightarrow[\sim]{i_{s+1}^*}
  H^{2p}(\mc{X}_{e(s)}, \mb{Z}) \xrightarrow[\sim]{(i_{s}^*)^{-1}} H^{2p}(\mc{X}_\infty, \mb{Z}). \]
   See \cite[p. $67,\, (2.4.13)$]{kuli} for further details.
Denote by $N:=-(1/2\pi i)\log(T)$. Using this operator $N$ we will recall the limit Hodge filtration.
   
\subsection{Limit Hodge filtration}\label{sec:lhf}
Denote by \[F_s^\bullet H^{2p}(\mc{X}_\infty, \mb{C}):=
(i_s^*)^{-1}(F^\bullet H^{2p}(\mc{X}_{e(s)}, \mb{C}))\]
the preimage of the 
Hodge filtration on $H^{2p}(\mc{X}_{e(s)}, \mb{C})$.
The dimension of $F^k_sH^{2p}(\mc{X}_\infty, \mb{C})$, 
denoted $m_k$, does not 
depend on the choice of $s \in \mf{h}$. Consider the 
Grassmann variety parameterizing $m_k$-dimensional
subspaces of $H^{2p}(\mc{X}_\infty, \mb{C})$, denoted $\mr{Grass}(m_k,H^{2p}(\mc{X}_\infty,
\mb{C}))$. There is a natural map:
\[\mf{h} \to \mr{Grass}(m_k,H^{2p}(\mc{X}_\infty,
\mb{C}))\, \mbox{ sending }s \in \mf{h}\, \mbox{ to }\mr{exp}(2\pi isN) F^k_s H^{2p}(\mc{X}_\infty, \mb{C}).\]
This map is invariant under the translation $s \mapsto s+1$ and tends 
to a limit $F^k H^{2p}(\mc{X}_\infty,\mb{C})$ as the imaginary
part of $s$ tends to $\infty$ i.e.,
\[F^k H^{2p}(\mc{X}_\infty, \mb{C}):=
 \lim_{\Ima(s) \to \infty} \mr{exp}(2\pi i sN)F^k_s 
 H^{2p}(\mc{X}_\infty, \mb{C}).\]
 See \cite[\S I.$2.6$]{kuli} or \cite[p. $254, 255$]{schvar}
 for further details.
Clearly,
 \begin{equation}\label{eq:hod}
  \lim_{\Ima(s) \to \infty} \mr{exp}(2\pi i sN)(F^p_s 
 H^{2p}(\mc{X}_\infty, \mb{C}) \cap H^{2p}(\mc{X}_\infty,\mb{Q}))
 \subset F^pH^{2p}(\mc{X}_\infty,\mb{C}) \cap H^{2p}(\mc{X}_\infty,\mb{Q}).
 \end{equation}
 This inclusion will play an important role in the definition of 
 the Mumford-Tate family in \S \ref{sec:mt}.
 
 \subsection{Limit weight filtration}
 One can observe that the decreasing filtration 
 \[F^0H^{2p}(\mc{X}_\infty, \mb{C}) \supseteq F^1H^{2p}(\mc{X}_\infty,
  \mb{C}) \supseteq ... \supseteq F^{2p}H^{2p}(\mc{X}_\infty,\mb{C})
  \supseteq 0
 \]
need not be a Hodge filtration i.e., $F^k \cap 
\ov{F}^{2p+1-k}$ need not be $0$. It was observed by Schmid that $H^{2p}(\mc{X}_\infty,\mb{Q})$ can be equipped with an increasing \emph{limit weight 
filtration} $W_\bullet$, arising from the monodromy action by $T$,
such that the two filtrations $F^\bullet$ and $W_\bullet$ together
define a mixed Hodge structure on $H^{2p}(\mc{X}_\infty,\mb{Q})$
(see \cite[Theorem $6.16$]{schvar}). Steenbrink in \cite{ste1}
retrieved the limit weight filtration using a spectral sequence.
We recall the $E_1$-terms of the spectral sequence:
 
 \begin{thm}[{\cite[Corollary $11.23$]{pet}}]\label{thm:spec}
  The spectral sequence 
  \[^{\infty} E_1^{p,q}:=\bigoplus\limits_{k \ge \max\{0,p\}}
   H^{q+2p-2k}(X(2k-p+1),\mb{Q})(p-k)  \]
   with the differential map $d:\, ^\infty E_1^{p-1,q} \to\, ^\infty E_1^{p,q}$ being a 
   combination of the restriction morphism and the Gysin morphism,
   degenerates at $E_2$. Moreover, $^{\infty} E_1^{p,q} \Rightarrow
H^{p+q}(\mc{X}_\infty, \mb{Q})$ with the weight filtration given by $^{\infty} E_2^{p,q} = 
\mr{Gr}^W_q H^{p+q}(\mc{X}_\infty, \mb{Q})$.
 \end{thm}

 \subsection{Specialization map}
%  
%  Since the central fiber $\mc{X}_0$ is the deformation 
%  retract of $\mc{X}$, the cohomology of both spaces
%  are the same i.e., $H^{2p}(\mc{X}_0,\mb{Z}) \cong 
%  H^{2p}(\mc{X},\mb{Z})$.
%   For any $s \in \Delta^*$, the natural inclusion 
%   from $\mc{X}_s$ to $\mc{X}$ induces a pull-back morphism 
%   from $H^{2p}(\mc{X},\mb{Z})$ to $H^{2p}(\mc{X}_s,\mb{Z})$.
%   The resulting morphism from $H^{2p}(\mc{X}_0,\mb{Z})$
%   to $H^{2p}(\mc{X}_s,\mb{Z})$ is known as the 
%   \emph{specialization morphism}. However, this morphism is 
%   not in general a morphism of mixed Hodge structures, if $H^{2p}(\mc{X}_s,\mb{Z})$
%   is equipped with the natural pure Hodge structure.
  
  By the identification between $H^{2p}(\mc{X}_\infty,\mb{Z})$ and $H^{2p}(\mc{X}_s,\mb{Z})$ mentioned above, we  
  get a specialization morphism (see \cite[\S $2$]{indpre}) which is a morphism 
  of mixed Hodge structures:
  \[\mr{sp}: H^{2p}(\mc{X}_0,\mb{Z}) \to H^{2p}(\mc{X}_\infty,\mb{Z}),\]
  where $H^{2p}(\mc{X}_\infty,\mb{Q})$ is equipped with the 
  limit mixed Hodge structure. Using the Mayer-Vietoris 
  sequence observe that the weight filtration on 
  $H^{2p}(\mc{X}_0,\mb{Q})$ arises from the spectral sequence
  with $E_1$-terms:
  \[E_1^{p,q}=H^q(X(p+1),\mb{Q})\, \Rightarrow H^{p+q}(\mc{X}_0,\mb{Q})\]
  where the differential $d: E_1^{p-1,q} \to E_1^{p,q}$ is 
  the restriction morphism (see \cite[Example $3.5$]{ste1}).
  Note that, the spectral sequence
  degenerates at $E_2$.
  
  \begin{rem}
   By the definition of $E_1^{j,q}$ and $^\infty E_1^{j,q}$ given above,
 we have a natural morphism from $E_1^{j,q}$ to 
 $^\infty E_1^{j,q}$, which commutes with the 
 respective differential maps $d$. As a result, 
this induces a morphism of spectral sequences:
\begin{equation}\label{eq:phi}
 \phi: E_2^{p,q} \to\,  ^\infty E_2^{p,q}.
\end{equation}
  \end{rem}

  We now compute the kernel over the
  weight graded pieces of the specialization morphism:
  
\begin{prop}\label{prop01}
 For $p \ge 0$, we have an exact sequence of the form:
 \[H^{q-2}(X(p+2),\mb{Q}) \to\, \,  E_2^{p,q} \xrightarrow{\phi}\, \,  ^\infty E_2^{p,q} \]
 where the first 
 morphism is induced by the Gysin morphism \[H^{q-2}(X(p+2),\mb{Q}) \to H^q(X(p+1),\mb{Q})=E_1^{p,q}\]
 and $\phi$ is as in \eqref{eq:phi}.
\end{prop}

\begin{proof}
 Note that the composed morphism 
 \[H^{q-2}(X(p+2),\mb{Q}) \to H^{q}(X(p+1),\mb{Q}) \to  H^{q}(X(p+2),\mb{Q})\, \mbox{ is the zero
  map}, \]
where the first morphism is simply the Gysin morphism and the second morphism is the restriction map.
Therefore, there is a natural map from 
$H^{q-2}(X(p+2),\mb{Q})$ to $E_2^{p,q}$. 
The difference between the spectral sequences $E_1^{p,q}$ and $^{\infty}E_1^{p,q}$ is that the 
differential map in the latter case also allows Gysin morphism. Therefore, the kernel of the 
morphism $\phi$ is isomorphic to the image of the Gysin map. This 
proves the proposition.
\end{proof}

%\begin{exa}
 %***The projective space trivially satisfies the Hodge conjecture**
 
 %***Every smooth projective variety satisfies the Hodge conjecture 
 %in weight $2$ (Lefschetz $(1,1)$-theorem)****
%\end{exa}

 \subsection{Flag Hilbert schemes}\label{sec:flag}
We refer the reader to \cite[\S $4.5$]{S1}
for a detailed study of flag Hilbert schemes.
Let \[\pi: \mc{X}_{\Delta^*} \to \Delta^*\] be a smooth, projective morphism 
over the punctured disc $\Delta^*$.
Fix a relative polarization $\mc{L}$ on $\mc{X}_{\Delta^*}$
inducing a closed immersion of $\mc{X}_{\Delta^*}$ into a 
relative projective space $\mb{P}^N_{\Delta^*}$ for 
some integer $N$. By the constancy of Hilbert 
polynomials in flat, projective families, every 
fiber of $\pi$ has the same Hilbert polynomial
(with respect to the polarization $\mc{L}$), say $Q$ 
(see \cite[Theorem III.$9.9$]{R1}). Recall, given a 
Hilbert polynomial $P$, there exists a projective scheme,
denoted $\mr{Hilb}_{P,Q}$, called a \emph{flag Hilbert scheme}
parameterizing pairs of the form $(Y \subset X 
\subset \mb{P}^N)$, where $Y$ (resp. $X$) is of 
Hilbert polynomial $P$ (resp. $Q$).

The flag Hilbert scheme 
$\mr{Hilb}_{P,Q}$ is equipped with an universal 
family $\mc{Y} \subset \mc{X}_{\mr{univ}}$ 
with $\mc{Y}, \mc{X}_{\mr{univ}}$ flat over 
$\mr{Hilb}_{P,Q}$ and for every $s \in \mr{Hilb}_{P,Q}$,
the corresponding fiber $\mc{Y}_s$ (resp. $\mc{X}_s$) has 
Hilbert polynomial $P$ (resp. $Q$) satisfying the 
universal property: if there exists a closed subscheme 
$\mc{Z} \subset \mc{X}_{\Delta^*}$, flat over $\Delta^*$ with 
fibers having Hilbert polynomial $P$, then there exists 
an unique morphism $f:\Delta^* \to \mr{Hilb}_{P,Q}$ such that 
the pull-back of the universal family 
$\mc{Y} \subset \mc{X}_{\mr{univ}}$
to $\Delta^*$ is isomorphic to $\mc{Z} \subset \mc{X}_{\Delta^*}$ (see \cite[Theorem $4.5.1$]{S1}).

 \begin{lem}\label{lem:hf}
  For every $0<\epsilon \in \mb{R}$ small enough, there exists 
$s_\epsilon \in \Delta^*$ of distance less than $\epsilon$ from 
the origin, such that every closed subvariety $Z_{s_\epsilon}$ of 
codimension $p$
in $\mc{X}_{s_\epsilon}$ extends to a $\Delta^*$-flat closed subscheme 
$\mc{Z} \subset \mc{X}_{\Delta^*}$ such that the fiber $\mc{Z} \cap 
\mc{X}_{s_\epsilon}$ over $s_\epsilon$ is isomorphic to $Z_{s_\epsilon}$.
\end{lem}

 \begin{proof}
  Since the Hilbert polynomial of the fibers of $\pi$ is $Q$, 
  by the universal property of 
Hilbert schemes there is a natural morphism 
\[f: \Delta^* \to \mr{Hilb}_Q\]
such that the pull-back of the universal family on 
$\mr{Hilb}_Q$ to $\Delta^*$ is 
isomorphic to $\mc{X}_{\Delta^*}$.
Let $S$ be the set of Hilbert polynomials $P$ of degree $n-p$
such that the image of the 
natural projection morphism from $\mr{Hilb}_{P,Q}$
to $\mr{Hilb}_Q$ does not contain the image of $f$ i.e., 
intersects properly the image of $f$. Clearly, $S$ is a countable set.
Note that the union of countably many proper closed subsets in $\Delta^*$ does not contain 
any open subsets. Hence,  
for every $0<\epsilon \in \mb{R}$ small enough, there exists 
$s_\epsilon \in \Delta^*$ of distance less than $\epsilon$ from 
the origin, such that 
 $f(s_\epsilon)$ does not lie in the image of the projection from 
 $\mr{Hilb}_{P,Q}$ to $\mr{Hilb}_Q$, as $P$ varies in the set $S$.
 In other words, every closed subscheme in $\mc{X}_{s_\epsilon}$ extends to 
 to a $\Delta^*$-flat closed subscheme of $\mc{X}_{\Delta^*}$.
This proves the lemma.
 \end{proof}

  \section{Mumford-Tate families}\label{sec:mt}
  In this section we introduce the concept of Mumford-Tate families. 
  These are smooth families of projective varieties such that the associated limit 
  mixed Hodge structure has ``as many'' Hodge classes as a general fiber in the family.
  The motivation behind the name is that Mumford-Tate groups are determined uniquely by the 
  set of Hodge classes in the associated tensor algebra.
      Let us first recall the definition of the Mumford-Tate group.
      
      \subsection{Mumford-Tate groups}\label{sec:mtg}
   Denote by $\mbb{S}$ the Weil restriction of scalars for the field extension $\mbb{C}/\mbb{R}$. 
Let $V$ be a $\mb{Q}$-vector space. 
   A pure Hodge structure of weight $n$ on $V$ is given by a non-constant homomorphism of
   $\mbb{R}$-algebraic groups
        \[ \phi: \mb{C}^*= \mbb{S}(\mb{R}) \to \mr{GL}(V)(\mb{R}) \]
   such that $\phi(r)=r^n\mr{Id}$ for all $r \in \mb{R}^* \subset \mbb{S}(\mb{R})=\mb{C}^*$. 
  Let $V_{\mb{C}}:= V \otimes_{\mb{Q}} \mb{C}$. To this group 
   homomorphism one associates the Hodge decomposition:
   \[V_{\mb{C}}= \bigoplus_{p+q = n} V^{p,q}\, \mbox{ where }\, V^{p,q}:= \{v \in V_{\mb{C}} |\,
     \phi(z)v = z^{p}\ov{z}^{q}v \mbox{ for all } z \in \mb{C}^*\}.\]

    The \emph{Mumford-Tate} group associated to the pure Hodge structure $(V, \phi)$, denoted $\MT(V,\phi)$, 
    is the smallest $\mbb{Q}$-algebraic subgroup of $\mr{GL}(V)$ whose set of real points 
    contain the image of $\phi$. 
    Denote by \[T^{m,n}(V) := V^{\otimes m}\otimes \mr{Hom}(V, \mbb{Q})^{\otimes n}.\]
    Note that, the Hodge structure on $V$ induces a pure Hodge structure on $T^{m,n}(V)$.
    Elements of \[F^{0}(T^{m,n}(V_{\mbb{C}})) \cap T^{m,n}(V)\] are called \emph{Hodge tensors}. 
   The Mumford-Tate group as the largest
   subgroup of $\mr{GL}(V_{\mbb{Q}})$ which fixes the Hodge tensors (see \cite[\S I.$B$]{ker}).

\begin{exa}\label{exofmtg}
We now recall some well-known examples of Mumford-Tate groups.
\begin{enumerate}
 \item Let $X$ be an abelian variety and $V = H^1(X,\mathbb{Q})$. 
 The Mumford-Tate group associated to the pure Hodge structure on $V$ will be denoted by $\mr{MT}(X)$.
 The polarization on $X$ corresponds to a non-degenerate alternating form $\phi: V \otimes V \to \mathbb{Q}$. Denote by $\mr{GSp}(V,\phi)$ the group of symplectic simplitudes with respect to the symplectic form $\phi$:
\[ \mr{GSp}(V,\phi) := \{ g \in \mr{GL}(V)  \ | \ \exists \  \lambda \in \mathbb{C}^{*} \ \mr{such \ that} \  \phi(gv,gw) = \lambda \phi(v,w) \  \forall  \ v, w \in V \}.\]
 Recall, for any abelian variety $X$, the Mumford-Tate group of $X$ is contained in the group of symplectic simplitudes i.e.  $\MT(X) \subseteq \GSp(V, \phi)$. An abelian variety is called \emph{simple} if it does not contain an abelian subvariety other than $0$ and $X$. If $X$ is simple and $\dim(X) = p$, where $p$ is a prime number, then $\MT(X) = \GSp(V, \phi)$.
 
\item Let $c$ be a positive integer. Let $X$ be a general complete intersection subvariety 
contained in $\mb{P}^{2m+c}$ of codimension $c$, for some $m \ge 1$. Assume that the degree of $X$ is at least $5$. Denote by $V:=H^{n}(X,\mathbb{Q})_{\mr{prim}}$ and $\phi: V \otimes V \to \mathbb{Q}$ the polarization on  $V$. Let $\go(V, \phi)$ be the group of orthogonal simplitudes with respect $\phi$:
\[\go(V, \phi) := \{ g \in \mr{GL}(V)  \ | \ \exists \  \lambda \in \mathbb{C}^{*} \ \mr{such \ that} \  \phi(gv,gw) = \lambda \phi(v,w) \  \forall  \ v, w \in V \}.  \] Then the Mumford-Tate group of $X$, $\MT(X) = \go(V,\phi)$. 

\end{enumerate}
\end{exa}

\subsection{Mumford-Tate families}\label{sec:mtf}
 Keep setup as in \S \ref{se01}. 
  Given any $s \in \mf{h}$, recall the exponential map $e$ from $\mf{h}$ to $\Delta^*$ and the 
  natural inclusion $i_s$ from  
  $\mc{X}_{e(s)}$ into $\mc{X}_\infty$.   
    Recall,
    \[\pi: \mc{X}_{\Delta^*} \to \Delta^*\] 
    the family of smooth, projective varieties. 
    For any $s \in \mf{h}$, $H^{2p}(\mc{X}_{e(s)},\mb{Q})$ is equipped with a natural pure 
    Hodge structure. Denote by $\MT_p(\mc{X}_{e(s)})$ the Mumford-Tate group associated to this 
    pure Hodge structure on  $H^{2p}(\mc{X}_{e(s)},\mb{Q})$. We say that $\pi$
    is a \emph{Mumford-Tate family of weight} $p$
  if for any class $\gamma \in F^pH^{2p}(\mc{X}_\infty, \mb{C}) \cap H^{2p}(\mc{X}_\infty, \mb{Q})$
  satisfying $N\gamma=0$, the pullback $i_s^*(\gamma) \in H^{2p}(\mc{X}_{e(s)}, \mb{Q})$ is fixed by 
  $\MT_p(\mc{X}_{e(s)})$ for a general $s \in \mf{h}$.
We say that  $\pi$ is \emph{Mumford-Tate}
 if it is Mumford-Tate of all weights.

\begin{exa}\label{hdegexam}
We now give some examples of Mumford-Tate families:
\begin{enumerate}
 \item By Lefschetz hyperplane section theorem, 
 for any smooth hypersurface $X$ in $\mb{P}^{2m}$ for $m \ge 2$, we have $H^{2p}(X,\mb{Q}) \cong \mb{Q}$
 for any $0 \le p \le 2m-1$. This implies 
 if  $\pi$ parametrizes smooth,
 hypersurfaces in $\mb{P}^{2m}$, then $\pi$ is 
 Mumford-Tate. 
 \item Let $\pi: \mc{X} \to \Delta$ be a smooth family of prime 
 dimensional abelian varieties such that the central fiber $\pi^{-1}(0)$ is simple. 
 Then $\pi$ is a Mumford-Tate family.  Indeed, since $\pi$ is a smooth family, the
 local system $\mb{V}_p:= R^{2p}\pi_*\mb{Q}$ has no monodromy 
 over the punctured disc. Hence, $H^{2p}(\mc{X}_\infty,\mb{Q}) \cong H^{2p}(\mc{X}_0,\mb{Q})$ as 
 pure Hodge structures, for all $p$ and the local system $\mb{V}_p$ is trivial. 
 By the same argument, $R^1\pi_*\mb{Q}$ is a trivial local system. A choice of the trivialization fixes
an identification:
\[ \psi_t: V_0 \xrightarrow{\sim} V_t,\, \mbox{ where } V_t:=H^1(\mc{X}_t,\mb{Q})\, \mbox{ for any } t 
 \in \Delta.\]
Note that the natural polarizations on $V_0$ and $V_t$ commutes with the identification $\psi_t$. 
This induces an isomorphism:
\begin{equation}\label{eq:ident}
 \mr{GSp}(V_t,\phi_t) \xrightarrow{\sim} \mr{GSp}(V_0,\phi_0)\, \mbox{ sending } 
 \left(V_t \xrightarrow[\sim]{g} V_t\right) \mbox{ to } \left(V_0 \xrightarrow[\sim]{\psi_t} V_t \xrightarrow[\sim]{g} V_t
 \xrightarrow[\sim]{\psi_t^{-1}} V_0\right).
\end{equation}
Now, $\gamma_0 \in H^{2p}(\mc{X}_\infty, \mb{Q})=H^{2p}(\mc{X}_0,\mb{Q})$ 
is a Hodge class if and only if it is fixed by the Mumford-Tate group $\mr{MT}(\mc{X}_0)$.
Since $\mc{X}_0$ is simple, $\MT(\mc{X}_0)=\mr{GSp}(V_0,\phi_0)$. 
Using the identification \eqref{eq:ident}, since the Hodge class
$\gamma_0$ is fixed by $\mr{GSp}(V_0,\phi_0)$, 
$i_s^*(\gamma)=\phi_s(\gamma)$ is fixed by $\mr{GSp}(V_s,\phi_s)$ for any $s \in \Delta^*$.
Since $\MT(\mc{X}_s)$ is contained in $\mr{GSp}(V_s,\phi_s)$, $\phi_s(\gamma)$ is fixed by $\MT(\mc{X}_s)$.
Hence, $\phi_s(\gamma)$ is a Hodge class in $H^{2p}(\mc{X}_s,\mb{Q})$. This proves the claim that $\pi$ 
is a Mumford-Tate family.

 \item Let $\pi: \mc{X} \to \Delta$ be a smooth 
 family of complex intersection subvarieties of codimension $c$ and let 
 $\pi^{-1}(0) = \mc{X}_{0}$. Suppose that 
 $\MT(\mc{X}_{0}) = \go(H^{n}(\mc{X}_0, \mathbb{Q})_{\mr{prim}}, \phi)$. Then $\pi$ is a Mumford-Tate family.
 The proof for this is the same as that of $(2)$ above with GSp replaced by GO.
 \end{enumerate}
 \end{exa}

 \begin{exa}(Examples of non Mumford-Tate families)\label{nonmtf}
 Recall for $d \ge 4$, the Noether-Lefschetz theorem states that a very general smooth, degree $d$ 
 surface in $\p3$ has Picard number $1$. The Noether-Lefschetz  
 locus parametrizes smooth degree $d$ surfaces in 
 $\p3$ with Picard number at least $2$. See  \cite{Dcont, D3, dan7} for some its geometric 
 properties. This means that 
 there are smooth families $\pi: \mc{X} \to \Delta$
 of hypersurfaces in $\p3$ such that $0 \in \Delta$ lies on the Noether-Lefschetz locus and 
 $\Delta^*$ does not intersect the Noether-Lefschetz locus. 
 Since $\pi$ is a smooth family, the local system $R^2\pi_*\mb{Q}$ does not have any monodromy over 
 the punctured disc. 
 Then, $H^2(\mc{X}_\infty, \mb{Q}) \cong H^2(\mc{X}_0.\mb{Q})$ as pure Hodge structures. In particular, 
 by the condition on the central fiber $\mc{X}_0$, 
 the rank of the Hodge lattice in $H^2(\mc{X}_\infty,\mb{Q})$ is at least $2$. But the rank of the 
 Hodge lattice in $H^2(\mc{X}_s,\mb{Q})$ is $1$ for any $s \in \Delta^*$. Since the pullback morphism 
 $i_s^*$ is an isomorphism, this implies that there is a Hodge class on $H^2(\mc{X}_\infty,\mb{Q})$ that 
 does not pullback to a Hodge class on $H^2(\mc{X}_s,\mb{Q})$. Hence, $\pi$ cannot be a Mumford-Tate family.
\end{exa}

\section{A cohomological version of the Hodge conjecture for singular varieties}\label{sec:shc}
In this section we define limit algebraic cohomology classes and limit Hodge classes. We show that the limit 
 algebraic cohomology classes are contained in the monodromy invariant limit Hodge classes and the converse holds for Mumford-Tate families. In subsection \ref{subsecOCG} and \ref{subsecB-G-S} we recall the necessary preliminaries for the Operational Chow group and the Bloch-Gillet-Soul\'{e} cycle class map. In \ref{subsecSHC} we state the Singular Hodge conjecture and in \ref{subsecSNCD} we show that the cohomology classes of algebraic cycles on a simple normal crossings variety are contained in the Hodge classes.
 
 We begin by recalling the classical Hodge conjecture.
 
 \subsection{The classical Hodge conjecture}\label{sec:cy}
Let $X$ be a smooth, projective variety. Given an integer 
$p>0$, denote by $Z^p(X)$ the free abelian group generated
by codimension $p$ algebraic subvarieties. There is a
natural \emph{cycle class map}:
\[\mr{cl}_p: Z^p(X) \to H^{2p}(X,\mb{Z})\]
which associates to an algebraic subvariety $W \subset X$
of codimension $p$, the fundamental class $[W] \in 
H^{2p}(X,\mb{Z})$ (see \cite[\S $11.1.2$]{v4} for further details)
and extend linearly.
Furthermore, by \cite[Proposition $11.20$]{v4}, the image of 
the cycle class map $\mr{cl}_p$ lies in 
$H^{p,p}(X,\mb{C}) \cap H^{2p}(X,\mb{Z})$ i.e., the cohomology
class of an algebraic variety is a Hodge class. Tensoring 
the cycle class map by rationals gives:
\[\mr{cl}_p: Z^p(X) \otimes_{\mb{Z}} \mb{Q} \to 
 H^{2p}(X,\mb{Q}) \cap H^{p,p}(X,\mb{C}).\]
We denote by $H^{2p}_{\mr{Hdg}}(X):=H^{2p}(X,\mb{Q}) \cap H^{p,p}(X,\mb{C})$ the
space of Hodge classes and the space of 
 algebraic classes $H^{2p}_A(X) \subset H^{2p}(X,\mb{Q})$ is the image of the (rational) cycle class map 
 $\mr{cl}_p$.
 The (rational) Hodge conjecture claims that the (rational) cycle class map $\mr{cl}_p$
 is surjective  for all $p$ i.e.,
 the natural inclusion $H^{2p}_A(X) \subset H^{2p}_{\mr{Hdg}}(X)$ is an equality for all $p$.

 \begin{defi}
 Let $X$ be a smooth, projective variety of dimension $n$.
 We say that $X$ \emph{satisfies} $\mr{HC}(p,n)$ if
 the natural inclusion  $H^{2p}_A(X) \subset H^{2p}_{\mr{Hdg}}(X)$ is an equality.
 We say that $X$ 
 \emph{satisfies the Hodge conjecture} if it satisfies $\mr{HC}(p,n)$ 
  for every $p \ge 0$. We say that $\mr{HC}(p,n)$ \emph{holds true} to mean that 
  every smooth, projective variety of dimension $n$ satisfies $\mr{HC}(p,n)$.
\end{defi}

 \subsection{Relative cycle class}\label{sec:relcc}
 Let 
 \[\pi: \mc{X}_{\Delta^*} \to \Delta^*\]
 be a smooth, projective morphism of relative dimension $n$.
Let $\mc{Z} \subset \mc{X}_{\Delta^*}$ be a closed 
subscheme of $\mc{X}_{\Delta^*}$, flat over $\Delta^*$
and of relative dimension $n-p$.
The fundamental class  
 of $\mc{Z}$ defines a global section $\gamma_{_{\mc{Z}}}$ of the local system $\mb{H}^{2p}:= R^{2p}\pi_*\mb{Z}$ 
 such that for every $t \in \Delta^*$, 
 the value $\gamma_{_{\mc{Z}}}(t) \in H^{2p}(\mc{X}_t,\mb{Z})$
 of $\gamma_{_{\mc{Z}}}$ at the point $t$
 is simply the fundamental class of $\mc{Z}_t:=\mc{Z} \cap \mc{X}_t$ in $\mc{X}_t$ (see \cite[\S $19.2$]{fult} and 
 \cite[\S B.$2.9$]{pet} for details). The pull-back of the local system $\mb{H}^{2p}$
under the exponential map $e:\mf{h} \to \Delta^*$ is a 
trivial local system with fiber $H^{2p}(\mc{X}_\infty, \mb{Z})$. The global section $\gamma_{_{\mc{Z}}}$ defines an element of $H^{2p}(\mc{X}_\infty, \mb{Z})$, which we again denote by $\gamma_{_{\mc{Z}}}$, such that for every $s \in \mf{h}$, the image $i_s^*(\gamma_{_{\mc{Z}}})$ is the fundamental class of $\mc{Z} \cap \mc{X}_{e(s)}$ in $\mc{X}_{e(s)}$, where $i_s$ is the natural inclusion of $\mc{X}_{e(s)}$ into $\mc{X}_\infty$.  
 
 \begin{defi}\label{defi:lim-alg}
 Denote by $H^{2p}_A(\mc{X}_\infty)$
the sub-vector space of $H^{2p}(\mc{X}_\infty, \mb{Q})$ 
generated by all such elements of the form $\gamma_{_{\mc{Z}}}$
arising from a $\Delta^*$-flat closed subscheme of 
relative dimension $n-p$ in $\mc{X}_{\Delta^*}$. We call 
$H^{2p}_A(\mc{X}_\infty)$ the \emph{limit 
algebraic cohomology group}. We define the 
\emph{limit Hodge cohomology group}
\[H^{2p}_{\mr{Hdg}}(\mc{X}_\infty):= F^p H^{2p}(\mc{X}_\infty, \mb{C}) \cap 
W_{2p} H^{2p}(\mc{X}_\infty, \mb{Q}).
\]
Note that, $H^{2p}_{\mr{Hdg}}(\mc{X}_\infty)$ need not be monodromy invariant. 
Recall, $N$ is a morphism of mixed Hodge structures from $H^{2p}(\mc{X}_\infty,\mb{Q})$ to 
$H^{2p}(\mc{X}_\infty,\mb{Q})(-1)$. 
We denote by 
$H^{2p}_{\mr{Hdg}}(\mc{X}_\infty)^{\mr{inv}}$ the monodromy invariant part of 
$H^{2p}_{\mr{Hdg}}(\mc{X}_\infty)$ i.e.,
\[ H^{2p}_{\mr{Hdg}}(\mc{X}_\infty)^{\mr{inv}}:= \ker\left(H^{2p}_{\mr{Hdg}}(\mc{X}_\infty)
\hookrightarrow  H^{2p}(\mc{X}_\infty,\mb{Q}) \xrightarrow{N}  
 H^{2p}_{\mr{Hdg}}(\mc{X}_\infty, \mb{Q})\right). \]
 \end{defi}

 We now prove that the limit algebraic cohomology group lies in the limit Hodge cohomology group.
This is the asymptotic version of a classical result in Hodge theory.

  \begin{prop}\label{prop02}
  The limit algebraic cohomology group is contained in the monodromy invariant part of the 
  limit Hodge cohomology group i.e., the natural inclusion 
  $H^{2p}_A(\mc{X}_\infty) \subset H^{2p}(\mc{X}_\infty,
  \mb{Q})$ factors through $H^{2p}_{\mr{Hdg}}(\mc{X}_\infty)^{\mr{inv}}$.
  \end{prop}

  \begin{proof}
   Take $\gamma \in H^{2p}_A(\mc{X}_\infty)$. 
   By construction, there exist 
   $\Delta^*$-flat closed subschemes $\mc{Z}_1,...,\mc{Z}_r$
   of relative dimension $n-p$ in $\mc{X}_{\Delta^*}$
   such that $\gamma=\sum a_i\gamma_{_{\mc{Z}_i}}$ for $a_i \in \mb{Q}$ and $\gamma_{_{\mc{Z}_i}} \in H^{2p}_A(\mc{X}_\infty)$
   is as defined above, arising from the fundamental
   class of $\mc{Z}_i$. 
   By construction, each $\gamma_{_{\mc{Z}_i}}$ arises from a 
   global section of the local system $\mb{H}^{2p}$.
   Hence, $\gamma_{_{\mc{Z}_i}}$ is monodromy invariant i.e.,
   $T(\gamma_{_{\mc{Z}_i}})=\gamma_{_{\mc{Z}_i}}$ for $1 \le i \le r$.
   This implies $N \gamma_{_{\mc{Z}_i}}=0$ for $1 \le i \le r$.
   
   As the cohomology class of $\mc{Z}_i \cap \mc{X}_{e(s)}$
   lies in $F^pH^{2p}(\mc{X}_{e(s)}, \mb{Q})$, we have 
   $\gamma_{_{\mc{Z}_i}} \in F^p_s H^{2p}(\mc{X}_\infty, \mb{Q})$
   for all $s \in \mf{h}$ (notations as in \S \ref{sec:lhf}).
   This implies $\gamma_{_{\mc{Z}_i}}$
   lies in $\mr{exp}(2\pi i sN)F^p_sH^{2p}(\mc{X}_\infty, \mb{Q})$ for every $s \in \mf{h}$.
   Recall from \S \ref{sec:lhf} that $F^pH^{2p}(\mc{X}_\infty, \mb{Q})$ contains 
   the limit of $\mr{exp}(2\pi i sN)F^p_sH^{2p}(\mc{X}_\infty, \mb{C})$ as $\mr{Im}(s)$ approaches $\infty$. Hence,
   $\gamma_{_{\mc{Z}_i}} \in F^pH^{2p}(\mc{X}_\infty, \mb{Q})$.
   As $\gamma_{_{\mc{Z}_i}}$ is monodromy invariant
   and a rational class,
   it must lie in $W_{2p} H^{2p}(\mc{X}_\infty, \mb{Q})$ (use the 
   invariant cycle theorem along with the fact that the degree $2p$ 
   cohomology of the central fiber is of weight 
   at most $2p$). 
   Therefore, $\gamma \in H^{2p}_{\mr{Hdg}}(\mc{X}_\infty)^{\mr{inv}}$.
   This proves the first part of the proposition.
  \end{proof}

  We now ask when is $H^{2p}_A(\mc{X}_\infty)$
  isomorphic to $H^{2p}_{\mr{Hdg}}(\mc{X}_\infty)^{\mr{inv}}$?
  One can naively guess that if the general fibers in the family $\pi$ satisfy the 
  Hodge conjecture then this happens. However, this is not enough (see Example \ref{nonmtf} above). 
  In particular, one needs to   additionally assume that the family $\pi$
   is Mumford-Tate. We prove:

  \begin{prop}\label{prop04}
  Suppose that $\pi$ is a Mumford-Tate family of weight $p$. 
  If a general fiber of $\pi$ satisfies HC$(p,n)$,
   then the inclusion from $H^{2p}_A(\mc{X}_\infty)$ 
   to $H^{2p}_{\mr{Hdg}}(\mc{X}_\infty)^{\mr{inv}}$ is an isomorphism.
\end{prop}

 Note that, by \emph{general} in the statement of the proposition,
we mean the complement of \emph{finitely many} proper, 
closed subvarieties 
of the punctured disc $\Delta^*$.

\begin{proof}
 We need to show that  every element in 
   $H^{2p}_{\mr{Hdg}}(\mc{X}_\infty)^{\mr{inv}}$ lies in $H^{2p}_A(\mc{X}_\infty)$.
   Since $\pi$ is a Mumford-Tate family, we have  
   \begin{equation}\label{eq01}
    H^{2p}_{\mr{Hdg}}(\mc{X}_\infty)^{\mr{inv}}=
 \lim_{\mr{Im}(s ) \to \infty} (F^p_sH^{2p}(\mc{X}_\infty,\mb{Q}) \cap H^{2p}(\mc{X}_\infty, \mb{Q})^{\mr{inv}}).
   \end{equation}
It therefore suffices to 
show that 
\[\lim_{\mr{Im}(s) \to \infty} (F^p_sH^{2p}(\mc{X}_\infty, \mb{Q}) \cap H^{2p}(\mc{X}_\infty, \mb{Q})^{\mr{inv}})\]
is contained in $H^{2p}_A(\mc{X}_\infty)$.

By Lemma \ref{lem:hf} for every $0<\epsilon \in \mb{R}$ small enough, there exists 
$s_\epsilon \in \Delta^*$ of distance less than $\epsilon$ from 
the origin, such that $\mc{X}_{s_\epsilon}$ satisfies $\mr{HC}(p,n)$ and 
 every closed subvariety $Z_{s_\epsilon}$ of 
codimension $p$
in $\mc{X}_{s_\epsilon}$ extends to a $\Delta^*$-flat closed subscheme 
$\mc{Z} \subset \mc{X}_{\Delta^*}$ such that the fiber $\mc{Z} \cap 
\mc{X}_{s_\epsilon}$ over $s_\epsilon$ is isomorphic to $Z_{s_\epsilon}$.
As observed before Definition \ref{defi:lim-alg}, the fundamental class of $\mc{Z}$ defines a 
section $\gamma_{_{\mc{Z}}} \in H^{2p}_A(\mc{X}_\infty)$
and is monodromy invariant.
Since $F^pH^{2p}(\mc{X}_{s_\epsilon}, \mb{Q})$
is isomorphic to $H^{2p}_A(\mc{X}_{s_\epsilon})$, this 
implies 
\[H^{2p}(\mc{X}_\infty,\mb{Q})^{\mr{inv}} \cap F^p_{s_\epsilon}H^{2p}(\mc{X}_\infty, \mb{Q})
 = (i_{s_\epsilon}^*)^{-1}(H^{2p}_A(\mc{X}_{s_{\epsilon}}))
 \subseteq H^{2p}_A(\mc{X}_\infty), \]
 where $i_{s_\epsilon}$ is the natural inclusion of 
 $\mc{X}_{e(s_\epsilon)}$ into $\mc{X}_\infty$.
Therefore, the limit as $\mr{Im}(s)$ tends to $\infty$, 
of $H^{2p}(\mc{X}_\infty,\mb{Q})^{\mr{inv}} \cap F^p_sH^{2p}(\mc{X}_\infty, \mb{Q})$ is contained in 
$H^{2p}_A(\mc{X}_\infty)$.  This proves the proposition.
\end{proof}

\subsection{Operational Chow group}\label{subsecOCG}
Let $Y$ be a quasi-projective variety (possibly singular),
of dimension say $n$. Consider a non-singular hyperenvelope of a 
compactification of $Y$ (see \cite[\S $1.4.1$]{soug}
for the definition and basic properties of hyperenvelopes). 
The hyperenvelope gives rise to a 
cochain complex of motives (see \cite[\S $2.1$]{soug}). 
For any positive integer $p$, one 
can then obtain an abelian group $R^0\mr{CH}^p(Y)$
arising as the cohomology group after applying the functor 
$\mr{CH}^p(-)$ to the cochain complex of motives (see 
\cite[\S $3.1.4$]{soug}). Observe that $R^0\mr{CH}^p(Y)$
does not depend on the choice of the compactification or the 
hyperenvelope. Note that,

\begin{thm}\label{th:chow}
 Fix a positive integer $p$.
 Then, the following holds true for $R^0\mr{CH}^p(Y)$:
 \begin{enumerate}
  \item if $Y$ is projective, then $R^0\mr{CH}^p(Y)$ is 
   the operational Chow group $A^p(Y)$ defined by 
  Fulton and MacPherson (see \cite[Chapter $17$]{fult}),
  \item if $Y$ is non-singular (but not necessarily projective),
  then $A^p(Y)$ is the free abelian group generated by the codimension 
  $p$ subvarieties in $Y$, upto rational equivalence,
  \item if $Y$ is non-singular and $\ov{Y}$ is a compactification 
  of $Y$ with boundary $Z:=\ov{Y} \backslash Y$, we then 
  have the exact sequence:
  \begin{equation}\label{eq:loc}
   0 \to R^0\mr{CH}^p(Y) \to R^0\mr{CH}^p(\ov{Y}) \to 
   R^0 \mr{CH}^p(Z)
  \end{equation}
  \item if $Y$ is the union of two proper closed subvarieties 
  $Y_1$ and $Y_2$, then we have the exact sequence:
  \begin{equation}\label{eq:may}
   0 \to R^0\mr{CH}^p(Y) \to R^0\mr{CH}^p(Y_1) \oplus R^0\mr{CH}^p(Y_2)
   \to R^0\mr{CH}^p(Y_1 \cap Y_2).
  \end{equation}
 \end{enumerate}
\end{thm}

\begin{proof}
 \begin{enumerate}
  \item This is \cite[Proposition $4$]{soug}.
  \item This is \cite[Proposition $17.3.1$ and 
  Corollary $17.4$]{fult}.
  \item This is \cite[Theorem $2(iii)$ and \S $3.1.1$]{soug}.
  \item This is \cite[Theorem $2(iv)$ and \S $3.1.1$]{soug}.
 \end{enumerate}
\end{proof}

\begin{note}
 If $Y$ is quasi-projective but not
 projective, we denote by 
 $A^p_c(Y):=R^0\mr{CH}^p(Y)$, the 
 \emph{compactly supported operational Chow 
cohomology}. Given any compactification 
 $\ov{Y}$ of $Y$, Theorem \ref{th:chow} implies that 
 we have the following exact sequence
 \begin{equation}\label{eq:comp}
 0 \to A^p_c(Y) \to A^p(\ov{Y}) \to A^p(\ov{Y}\backslash Y)  
 \end{equation}
For $Y$ a projective variety, there are natural functorial cycle class maps 
(see \cite{soub} or \cite[\S $2$]{latv}):
\[\mr{cl}_p: A^p(Y) \to \mr{Gr}^W_{2p} H^{2p}(Y,\mb{Q}) \mbox{ and }
 \mr{cl}_p^c: A^p_c(Y_{\mr{sm}}) \to \mr{Gr}^W_{2p} H^{2p}_c(Y_{\mr{sm}},\mb{Q})\]
 which agree with the usual cycle class map 
 (see \cite[\S $11.1.2$]{v4}) if $Y$ is non-singular
 (here $Y_{\mr{sm}}$ denotes the smooth locus of $Y$).
 For $Y$ projective, define the \emph{algebraic cohomology group}
 denoted by $H^{2p}_A(Y) \subset \mr{Gr}^W_{2p}H^{2p}(Y,\mb{Q})$
 to be the image of the cycle class map $\mr{cl}_p$.
 \end{note}

 \subsection{Bloch-Gille-Soul\'{e} Cycle class map}\label{subsecB-G-S}
Let $Y$ be a scheme and $\phi :U \to Y$, $\gamma: V \to U \times_Y U$ be envelopes. Let $p_i: V \to U$ denote the compositions of $\gamma$ with the projections $U \times_{Y} U \to U$.

\begin{thm}({\cite[Theorem A.3]{soub}})
 There is a left-exact sequence of Chow cohomology groups
 \[ 0 \to CH^{*}(Y)\xrightarrow{\phi^{*}} CH^{*}(U)\xrightarrow{p_1^{*}-p_2^{*}} CH^{*}(V). \]
\end{thm}

Using the cycle map over smooth, quasi-projective varieties $U$ and $V$, Bloch-Gillet-Soul\'{e} uses the above theorem to conclude:

\begin{cor}({\cite[Corollary A.4]{soub}})
 On the category of varieties over $\mb{C}$, there is a ``cycle class'' natural transformation of contravariant functors to the category of commutative, graded rings:
  \[ \bigoplus_{p} \mr{cl}_p : \bigoplus_{p}CH^{p}(-) \to \bigoplus_{p} Gr^{W}_{0}H^{2p}(- \ , \mathbb{Q}(p)).\]
\end{cor}

 \subsection{Singular Hodge conjecture}\label{subsecSHC}
 
 We are now ready to give a  formulation of the Hodge conjecture for singular varieties. 
   Let $Y$ be a projective variety of dimension $n$. Fix a positive integer $p \le n$. 
  We say that $Y$ \emph{satisfies} $\mr{SHC}(p,n)$ if 
  the singular locus of $Y$ is of dimension at most $p-1$ and 
  the algebraic cohomology group $H^{2p}_A(Y)$ coincides with 
  \[H^{2p}_{\mr{Hdg}}(Y):= \mr{Gr}^W_{2p} H^{2p}(Y,\mb{Q}) \cap F^{2p} \mr{Gr}^W_{2p} H^{2p}(Y,\mb{C}).\]

 In the case when $Y$ is non-singular and projective, this simply is  
 the classical Hodge conjecture (in weight $p$), which we already denote 
 by $\mr{HC}(p,n)$.

 \subsection{Algebraic cycles on simple normal crossings divisors}\label{subsecSNCD}
 We now prove that the cohomology classes of algebraic cycles on a simple normal crossings variety are
 Hodge classes. This is a generalization to the singular case of a classical result in Hodge theory.
 Recall, $\mc{X}_0$ is called a \emph{simple normal crossings variety} if 
 $\mc{X}_0$ is connected, $\mc{X}_0= X_1 \cup ... \cup X_r$ with $X_i$ irreducible, non-singular
 for all $i$ and the intersection of any $p$ of the irreducible components of $\mc{X}_0$ is 
 non-singular of codimension $p$, for any $p \ge 1$.

 \begin{lem}\label{lem01}
 Let $\mc{X}_0$ be a simple normal crossings variety. Then,
 the cycle class map $\mr{cl}_p$ from $A^p(\mc{X}_0)$ to 
 $\mr{Gr}^W_{2p} H^{2p}(\mc{X}_0,\mb{Q})$ factors through 
 \[H^{2p}_{\mr{Hdg}}(\mc{X}_0):= F^p \mr{Gr}^W_{2p} H^{2p}(\mc{X}_0,
 \mb{C}) \cap \mr{Gr}^W_{2p} H^{2p}(\mc{X}_0,\mb{Q}).\]
\end{lem}

\begin{proof} 
We use recursion on the components of $\mc{X}_0$. 
Let $X_0,...,X_r$ be the irreducible components of $\mc{X}_0$.
Denote by $Z_i:= \ov{\mc{X}_0 \backslash (X_1 \cup ... \cup X_i)}$, the complement of the components $X_1, ..., X_i$ for $i \ge 1$. Let $Z_0:= \mc{X}_0$. Since $X_i$, $X_j$ and $X_i \cap X_j$ are non-singular for all $i,j$, they have pure Hodge structures. 
 Moreover by \cite[Theorem $5.39$]{pet},  $H^{2p-1}(X_i \cap Z_i, \mb{Q})$ is of weight at most $2p-1$ i.e., 
$\mr{Gr}^W_{2p} H^{2p-1}(X_i \cap Z_i, \mb{Q})=0$ for all 
$1 \le i \le r-1$. Therefore for all $1 \le i \le r-1$, we 
have the following exact sequence of pure Hodge structures:
\begin{equation}\label{eq02}
  0 \to \mr{Gr}^W_{2p} H^{2p}(Z_{i-1},\mb{Q}) \to H^{2p}(X_i,\mb{Q})
 \oplus \mr{Gr}^W_{2p}H^{2p}(Z_i,\mb{Q}) \to 
 \mr{Gr}^W_{2p}H^{2p}(X_i \cap Z_i, \mb{Q})
\end{equation}
 Moreover, by Theorem \ref{th:chow}, we have the exact sequence:
 \begin{equation}\label{eq02a}
   0 \to A^p(Z_{i-1}) \to A^p(X_{i})
 \oplus A^p(Z_{i}) \to 
 A^p(X_{i} \cap Z_{i})
 \end{equation}
 By the functoriality of the cycle class maps $\mr{cl}_p$, we have the following diagram

    {\[\begin{diagram}
  0 &\rTo& A^p(Z_{i-1}) &\rTo & A^p(X_{i})
  \oplus A^p(Z_{i}) &\rTo &
  A^p(X_{i} \cap Z_{i}) \\
  & & \dTo^{\mr{cl}_p} &  &\dTo^{\mr{cl}_p} &  &\dTo^{\mr{cl}_p} & & \\
   0 &\rTo& \mr{Gr}^W_{2p} H^{2p}(Z_{i-1},\mb{Q}) &\rTo & H^{2p}(X_{i},\mb{Q})
  \oplus \mr{Gr}^W_{2p}H^{2p}(Z_{i},\mb{Q}) &\rTo & 
  \mr{Gr}^W_{2p}H^{2p}(X_{i} \cap Z_{i}, \mb{Q})
   \end{diagram}\]}
For the base case, consider $i=r-1$. Note that, $Z_{r-1} = X_r$. 
 Since $X_r$ is non-singular, $A^p(Z_{r-1})$ is the usual Chow group.
 Therefore, $\mr{cl}_p(A^p(Z_{r-1})) \subset H^{2p}_{\mr{Hdg}}(Z_{r-1})$.

 Now for the recursion step. Assume that $\mr{cl}_p(A^p(Z_{i})) \subset H^{2p}_{\mr{Hdg}}(Z_i)$.
 Since the exact sequence \eqref{eq02} is a morphism of pure Hodge structures,
  the commutativity of the left hand square implies that 
$\mr{cl}_p(A^p(Z_{i-1})) \subset H^{2p}_{\mr{Hdg}}(Z_{i-1})$. 
This proves the lemma.
\end{proof}

\section{Main results}\label{sec:main}
In this section we introduce the concept of MT-smoothable varieties. Consider a simple normal crossings variety $X$ 
(in the sense of \S \ref{subsecSNCD}). Denote by $X(2)$ the disjoint union of intersection of any $2$ 
 irreducible components of $X$. We prove that if $X$ is MT-smoothable and $X(2)$ satisfies $\mr{HC}(p-1,n-1)$
 then $X$ satisfies $\mr{SHC}(p,n)$
 (see Theorem \ref{th01}). This is a generalization of Theorem \ref{thmintro1} in the introduction.
 Moreover, if there is an irreducible component $X_i$ of $X$ such that 
the restriction morphism on cohomology is surjective, then $X_i$ satisfies the classical Hodge conjecture (see 
Corollary \ref{cor02}). Finally, if the variety has worse singularities than simple normal crossings, then 
 we reduce the singular Hodge conjecture to a question solely on the algebraic classes (see Theorem \ref{th04}).

\begin{defi}
 Let $X$ be a singular projective variety of dimension $n$ 
and $p$ be an integer such that $\dim(X_{\mr{sing}}) \le p-1$.
We say that $X$ is \emph{MT-smoothable of weight} $p$ if 
there exists a flat, projective, Mumford-Tate family 
\[\pi_0: \mc{Y} \to  \Delta\]
smooth over $\Delta^*$, containing $X$ as a central fiber and 
a general fiber satisfying $\mr{HC}(p,n)$. We call $\pi_0$ a \emph{MT-smoothing of weight} $p$ 
of $X$.
\end{defi}

 Given a normal crossings variety $X$,  We prove: 

\begin{thm}\label{th01}
 Let $X$ be a simple normal crossings variety of dimension $n$. Assume that
 every irreducible component of $X(2)$ satisfies  $\mr{HC}(p-1,n-1)$. If $X$ is MT-smoothable of weight $p$, then 
 $X$ satisfies $\mr{SHC}(p,n)$ i.e., 
 \[H^{2p}_A(X,\mb{Q}) \cong H^{2p}_{\mr{Hdg}}(X,\mb{Q}).\]
 Moreover, for every irreducible component $X_i$ of $X$,
 the image of the restriction morphism from $H^{2p}_\mr{Hdg}(X, \mb{Q})$ to 
 $H^{2p}_{\mr{Hdg}}(X_i,\mb{Q})$ are cohomology classes of algebraic
 cycles i.e., the image 
 \[\mr{Im}(H^{2p}_\mr{Hdg}(X, \mb{Q}) \to H^{2p}_{\mr{Hdg}}(X_i,\mb{Q}))\]
 is contained in 
 $H^{2p}_A(X_i,\mb{Q})$.
\end{thm}

\begin{proof}
Since $X$ is MT-smoothable of weight $p$, there exists a Mumford-Tate family of weight $p$ 
\[\pi: \mc{X} \to \Delta\]
with central fiber $X$ and general fibers satisfying $\mr{HC}(p,n)$. 
By Proposition \ref{prop04} and Lemma \ref{lem01},
we have a morphism $\mr{sp}_{_A}$ from $H^{2p}_A(X)$ to $H^{2p}_A(\mc{X}_\infty)$ given by the composition:
\[\mr{sp}_{A}: H^{2p}_A(X) \hookrightarrow H^{2p}_{\mr{Hdg}}(X)
 \xrightarrow{\mr{sp}} H^{2p}_{\mr{Hdg}}(\mc{X}_\infty)^{\mr{inv}} \cong 
 H^{2p}_A(\mc{X}_\infty).\]
We claim that $\mr{sp}_{_A}$ is surjective. 
Recall from Definition \ref{defi:lim-alg}, $H^{2p}_A(\mc{X}_\infty)$ is generated as a $\mb{Q}$-vector space by classes $\gamma_{_{\mc{Z}}}$ where $\mc{Z} \subset \mc{X}_{\Delta^*}$ is a $\Delta^*$-flat closed 
subscheme of relative dimension 
$n-p$. 

Denote by $\ov{\mc{Z}}$ the closure of $\mc{Z}$ in $\mc{X}$. By \cite[\S $6.1$]{fult}, the intersection product $\ov{\mc{Z}}.X_i$ of $\ov{\mc{Z}}$
with $X_i$ is of codimension $p$ in $X_i$ .
Denote by $\gamma_i \in H^{2p}(X_i,\mb{Q})$ the cohomology class of 
the intersection product $\ov{\mc{Z}}.X_i$ for $1 \le i \le r$.
By the associativity of intersection product 
(see \cite[Proposition $8.1.1$ or Proposition $8.3$]{fult}), for any pair of integers $1 \le i <j \le r$, the image of 
$\gamma_i$ (resp. $\gamma_j$) under the restriction morphisms
from $H^{2p}(X_i, \mb{Q})$ (resp. $H^{2p}(X_j, \mb{Q})$)
to $H^{2p}(X_i \cap X_j, \mb{Q})$ coincides.
Using \eqref{eq02}  one can 
observe that there exists an algebraic cohomology class 
$\gamma \in H^{2p}_A(X)$ such that the image of $\gamma$ under 
the restriction morphism from $H^{2p}_A(X)$ to $H^{2p}_A(X_i)$
is $\gamma_i$ for $1 \le i \le r$. In other words, the cohomology class
of $\ov{\mc{Z}}$ in $H^{2p}(\mc{X}, \mb{Q})$ (see \cite[\S B.$2.9$]{pet})
pulls back to $\gamma$ in $H^{2p}(X, \mb{Q})$ and to 
the cohomology class $[\mc{Z} \cap \mc{X}_t] \in 
H^{2p}(\mc{X}_t, \mb{Q})$ over $\mc{X}_t$, for any $t \in \Delta^*$.
 This means that under the specialization morphism $\mr{sp}$
from $H^{2p}(X, \mb{Q})$ to $H^{2p}(\mc{X}_\infty, \mb{Q})$,
$\gamma$ maps to $\gamma_{_{\mc{Z}}}$. This proves our claim.

% There exists an element 
% $\gamma_{_{\mc{Z}}} \in H^{2p}_A(\mc{X}_\infty)$ such that the 
% pull-back of $\gamma_{_{\mc{Z}}}$ under the inclusion 
% $i_s: \mc{X}_{e(s)} \hookrightarrow \mc{X}_\infty$ is the 
% cohomology class $[\mc{Z} \cap \mc{X}_{e(s)}] \in H^{2p}(\mc{X}_{e(s)},
% \mb{Q})$ (see \S \ref{sec:rel}).
% 
% So, it suffices to show that there exists an element 
% $\gamma \in H^{2p}_A(\mc{X}_0)$ which maps to $\gamma_{_{\mc{Z}}}$
% under the above morphism $\mr{sp}_{_A}$.

By Proposition \ref{prop01}, the kernel of the specialization morphism
\[\mr{Gr}^W_{2p}H^{2p}(X, \mb{Q}) = E_2^{0,2p} \xrightarrow{
 \mr{sp}}\, ^{\infty}E_2^{0,2p}=\mr{Gr}^W_{2p}H^{2p}(\mc{X}_\infty, \mb{Q}) \]
is isomorphic to the image of the Gysin morphism from 
$H^{2p-2}(X(2), \mb{Q})$ to $H^{2p}(X, \mb{Q})$ (as $X(2)$ is non-singular, 
$H^{2p-2}(X(2),\mb{Q})$ has a pure Hodge structure of weight $2p-2$).
By assumption, every irreducible component of 
$X(2)$ satisfies HC$(p-1,n-1)$.
Then, we get the following commutative diagram of exact 
sequences:
\[\begin{diagram}
   H^{2p}_A(X(2))&\rTo&H^{2p}_A(X)&\rTo^{\mr{sp}_{_A}}&H^{2p}_A(\mc{X}_\infty)&\rTo&0\\
   \dTo^{\cong}&\circlearrowleft&\dInto&\circlearrowleft&
   \dTo^{\cong}\\
   H^{2p}_{\mr{Hdg}}(X(2))&\rTo&H^{2p}_{\mr{Hdg}}(X)&\rTo^{\mr{sp}}&H^{2p}_{\mr{Hdg}}(\mc{X}_\infty)^{\mr{inv}}&
     \end{diagram}\]
By diagram chase (or using four lemma for the diagram of
exact sequences), we conclude that the middle 
morphism from $H^{2p}_A(X)$ to $H^{2p}_{\mr{Hdg}}(X)$ is surjective, hence an isomorphism.
This proves the first part of the theorem. The second
part of the theorem follows immediately from the 
following commutative diagram, which arises 
from the Mayer-Vietoris sequence:
 \[\begin{diagram}
    H^{2p}_A(X)&\rInto&H^{2p}_A(X_i) \oplus 
    H^{2p}_A(\ov{X \backslash X_i})\\
    \dTo^{\cong}&\circlearrowleft&\dInto\\
    H^{2p}_{\mr{Hdg}}(X)&\rInto&H^{2p}_{\mr{Hdg}}(X_i) \oplus 
      H^{2p}_{\mr{Hdg}}(\ov{X \backslash X_i})
  \end{diagram}\]
This proves the theorem.
\end{proof}

\begin{cor}\label{cor02}
 Notations and hypothesis 
 as in Theorem \ref{th01}. Let $X_1$ be an irreducible component in $X$ such that 
 the complement $X_1^c:= \ov{X\backslash X_1}$ (the closure of $X\backslash X_1$
 in $X$) satisfies:
 \begin{equation}\label{eq:cond}
  \mr{Im}(H^{2p}_{\mr{Hdg}}(X_1) \to H^{2p}_{\mr{Hdg}}(X_1^c \cap X_1)) \subset 
  \mr{Im}(H^{2p}_{\mr{Hdg}}(X_1^c) \to H^{2p}_{\mr{Hdg}}(X_1^c \cap X_1)).
 \end{equation}
Then, $X_1$ satisfies $\mr{HC}(p,n)$.
\end{cor}

\begin{proof}
 Using the Mayer-Vietoris sequence we have the following commutative diagram:
 \[\begin{diagram}
   0&\rTo&H^{2p}_A(X)&\rTo&H^{2p}_A(X_1) \oplus H^{2p}_A(X_1^c)& &\\
   & &\dTo^{\cong}&\circlearrowleft&\dInto& &
   \\
   0&\rTo&H^{2p}_{\mr{Hdg}}(X)&\rInto&H^{2p}_{\mr{Hdg}}(X_1) \oplus H^{2p}_{\mr{Hdg}}(X_1^c)&\rTo&H^{2p}_{\mr{Hdg}}(X_1^c \cap X_1)&
     \end{diagram}\]
     where the isomorphism of the first vertical arrow follows from Theorem \ref{th01}
     and the bottom row is exact.
     If \eqref{eq:cond} is satisfied then for any $\gamma \in H^{2p}_{\mr{Hdg}}(X_1)$, there 
     exists $\gamma' \in H^{2p}_{\mr{Hdg}}(X_1^c)$ such that their restrictions to $X_1 \cap X_1^c$
     agree. In other words, $\gamma \oplus \gamma'$ maps to zero in $H^{2p}_{\mr{Hdg}}(X_1^c \cap X_1)$.
     By diagram chase, one observes that there exists $\gamma_A \in H^{2p}_A(X_1)$ which maps to 
     $\gamma$. This proves $H^{2p}_A(X_1) \cong H^{2p}_{\mr{Hdg}}(X_1)$. In other words, $X_1$
     satisfies $\mr{HC}(p,n)$. This proves the corollary.
\end{proof}

One immediately asks whether there are examples where \eqref{eq:cond} is satisfied? 

\begin{exa}
 Let $X$ be a projective variety of dimension $n$
 with only ordinary double point singularities.
 Suppose also that $X$ is smoothable. Then, there 
 exists a flat, projective family 
 \[\pi_0: \mc{Y} \to \Delta\]
 smooth over $\Delta^*$, $X$ as the central fiber and $\mc{Y}$ is a regular variety.
 Moreover, there exists a semi-stable reduction of $\pi_0$:
 \[\pi: \mc{X} \to \Delta\]
 such that the central fiber $\mc{X}_0:= \wt{X} \cup E$, where $E$ is a disjoint union of 
 quadric hypersurfaces in $\mb{P}^{n+1}$ and $E \cap \wt{X}_0$ is the intersection of $E$
 by hyperplanes in copies of $\mb{P}^{n+1}$. If $n=2p$ for some $p$,
 then the $n$-th rational cohomology of a 
 quadric hypersurface in $\mb{P}^n$ is isomorphic to $\mb{Q}$. This implies the natural 
 restriction morphism from $H^{2p}(E)$ to $H^{2p}(E \cap \wt{X})$ is surjective. In this case, 
 taking $X_1:= \wt{X}$, \eqref{eq:cond} is satisfied.
\end{exa}

A natural conjecture arises from our observations:

\begin{conj}
 Let $X$ be a singular projective variety, $\phi: \wt{X} \to X$ be any resolution of 
 singularities and $E$ be the exceptional divisor. Let $p$ be an integer such that 
 $\dim(X_{\mr{sing}}) \le p-1$. We then have an exact sequence on cohomology
 (see \cite[Corollary-Definition $5.37$]{pet}):
 \[ H^{2p}(X) \to H^{2p}(\wt{X}) \to H^{2p}(E)\]
  We conjecture that taking 
  algebraic cohomology groups preserves the exactness of the sequence
  i.e., the following sequence is exact:
 \[ H^{2p}_A(X) \to  H^{2p}_A(\wt{X}) \to H^{2p}_A(E).\]
\end{conj}

We now observe that this conjecture is closely related to the singular Hodge conjecture 
(which is equivalent to the Hodge conjecture).

\begin{thm}\label{th04}
Let $X$ be a singular projective variety of dimension $n$ 
and $p$ be an integer such that $\dim(X_{\mr{sing}}) \le p-1$.
If $X$ satisfies $\mr{SHC}(p,n)$, then $X$ satisfies Conjecture A. 
Conversely, if $\mr{HC}(p-1,n-1)$ holds true, 
$X$ is MT-smoothable of weight $p$ and satisfies Conjecture A, then $X$ 
satisfies $\mr{SHC}(p,n)$.
\end{thm}

\begin{proof}
 If $X$ satisfies the $\mr{SHC}(p,n)$, then $H^{2p}_A(X) \cong H^{2p}_{\mr{Hdg}}(X)$. 
 Let \[\phi: \wt{X} \to X\] be a resolution of $X$ and $E$ be the exceptional divisor. 
 We then have the following commutative diagram:
 \begin{equation}\label{eq:diag01}
  \begin{diagram}
   H^{2p}_A(X)&\rTo&H^{2p}_A(\wt{X})&\rTo&H^{2p}_A(E)\\
   \dTo^{\cong}&\circlearrowleft&\dInto&\circlearrowleft&
   \dInto\\
   H^{2p}_{\mr{Hdg}}(X)&\rInto&H^{2p}_{\mr{Hdg}}(\wt{X})&\rTo&H^{2p}_{\mr{Hdg}}(E)&
     \end{diagram}
 \end{equation}
     where the bottom row is exact, injective on the left and the top row is a complex. 
     To prove Conjecture A, we need to show that the top row is exact in the middle. 
     For this, take $\gamma \in H^{2p}_A(\wt{X})$ which maps to zero in $H^{2p}_A(E)$. By diagram 
     chase it is easy to check that there exists $\gamma' \in H^{2p}_A(X)$ which maps to $\gamma$.
     In other words, the top row of \eqref{eq:diag01} is exact in the middle.
     This proves the first part of the theorem.
     
     We now assume that $X$ satisfies 
     Conjecture A. Let $\pi_0: \mc{Y} \to \Delta$ be a MT-smoothing of weight $p$ 
     of $X$. 
  By the semi-stable reduction theorem (see \cite[Chapter II]{kempmum}) 
  there exists a flat, projective family  
  $\pi: \mc{X} \to \Delta$
  which has the same fiber over $\Delta^*$ as $\pi_0$, $\mc{X}$
  is regular,
  the central fiber $\mc{X}_0$ is a reduced simple normal crossings divisor 
  with one of the irreducible components, say $\widetilde{X}$ being
  proper birational to $X$. Furthermore, 
  the complement $\wt{X}^c:=\ov{\mc{X}_0 \backslash \wt{X}}$ satisfies:
  \[ \mc{X}_0 \backslash \wt{X}^c \cong \wt{X} \backslash (\wt{X}^c \cap \wt{X}) \cong X \backslash X_{\mr{sing}}\]
  i.e., $\mc{X}$ is isomorphic to $\mc{Y}$ away from $X_{\mr{sing}}$.
  Using the Mayer-Vietoris sequence and Conjecture A we have the following commutative diagram
  of exact sequences:
  \begin{equation}\label{eq:diag02}
   \begin{diagram}
       H^{2p}_A(X)&\rTo&H^{2p}_A(\wt{X})&\rTo&H^{2p}_A(\wt{X} \cap \wt{X}^c)\\
   \dInto&\circlearrowleft&\dInto&\circlearrowleft&
   \dTo^{\cong}\\
   H^{2p}_A(\mc{X}_0)&\rTo&H^{2p}_A(\wt{X}) \oplus H^{2p}_A(\wt{X}^c)&\rTo&H^{2p}_A(\wt{X} \cap \wt{X}^c)
         \end{diagram}
  \end{equation}
where the first vertical morphism is induced by the pullback from $X$ to $\mc{X}_0$ and the second 
one is the natural inclusion. By snake lemma, this gives rise to the exact sequence:
\begin{equation}\label{eq16}
 0 \to H^{2p}_A(X) \to H^{2p}_A(\mc{X}_0) \to H^{2p}_A(\wt{X}^c)
\end{equation}
  Since $X_{\mr{sing}}$ is of 
  dimension at most $p-1$, $H^{i}(X_{\mr{sing}})=0$ for $i \ge 2p-1$.
  Then, the 
  long exact sequences in cohomology associated to the pairs 
  $(X,X_{\mr{sing}})$ and $(\mc{X}_0, \wt{X}^c)$ (see \cite[Proposition 
  $5.46$ and Corollary B.$14$]{pet})) implies
$ \mr{Gr}^W_{2p} H^{2p}_c(U) \cong \mr{Gr}^W_{2p} H^{2p}(X)$ where  $U:= X \backslash X_{\mr{sing}}$. Furthermore, 
\[0 \to \mr{Gr}^W_{2p} H^{2p}_c(U, \mb{Q}) \to \mr{Gr}^W_{2p} H^{2p}(\mc{X}_0, \mb{Q}) \to \mr{Gr}^W_{2p} H^{2p}(\wt{X}^c,\mb{Q})\]
is an exact sequence of pure Hodge structures.
This gives rise to the exact sequence:
      \begin{equation}\label{eq15}
      0 \to H^{2p}_{\mr{Hdg}}(X) \to H^{2p}_{\mr{Hdg}}(\mc{X}_0)
      \to H^{2p}_{\mr{Hdg}}(\wt{X}^c)
  \end{equation}
  of $\mb{Q}$-vector spaces. 
  Then, there is a natural morphism of exact sequences from \eqref{eq16} to \eqref{eq15}:
 \[\begin{diagram}
 0 &\rTo & H^{2p}_A(X) & \rTo & H^{2p}_A({\mc{X}}_0) &\rTo & H^{2p}_A(\wt{X}^c) \\
 & & \dInto & \circlearrowleft  &\dTo^{\cong} & \circlearrowleft &\dInto & & \\
    0 &\rTo& H^{2p}_{\mr{Hdg}}(X) &\rTo  & H^{2p}_{\mr{Hdg}}(\mc{X}_0) &\rTo & H^{2p}_{\mr{Hdg}}(\wt{X}^c)
   \end{diagram}\]
   where the isomorphism of the middle vertical arrow follows from Theorem \ref{th01}.
   Applying snake lemma once again we conclude that the first vertical morphism is surjective.
 In other words, $X$ satisfies $\mr{SHC}(p,n)$. This proves the converse and hence the theorem.
\end{proof}

 %\begin{rem}
  %A statement analogous to Corollary \ref{cor01} also holds true in 
  %the case $n$ is even with the additional assumption that $X$ 
  %lies in a Mumford-Tate family of weight $n/2$.
 %\end{rem}

 \section{Examples of Mumford-Tate families}\label{sec:exa}
 
 In \S \ref{sec:mt} we introduced Mumford-Tate families. For such families, 
 the central fiber displays interesting properties.
 For example, if the central fiber is smooth, then it is easy to check that it 
 satisfies the Hodge conjecture if a general fiber satisfies the Hodge conjecture. More generally, if the central fiber is a 
 reduced, simple normal crossings divisor, then 
 it satisfies the singular Hodge conjecture if the general fiber satisfies the Hodge conjecture (see Theorem \ref{th01}). 
 In this section we use correspondences to give a general method to produce Mumford-Tate families (see Theorem \ref{thm:mt}).
 We give examples in Corollary \ref{cor:exa}.

\subsection{Strict Mumford-Tate families}\label{sec:str}
Let $\pi_1: \mc{X}^* \to \Delta^*$ be a smooth, projective morphism over the punctured disc $\Delta^*$. 
Recall that $\pi_1$ is called a Mumford-Tate family if the pullback of every monodromy invariant Hodge class on 
$H^{2p}(\mc{X}_\infty, \mb{Q})$ to a general fiber is fixed by the associated Mumford-Tate group, for every $p$.
Here we generalize this condition to the tensor algebra of the cohomology ring $H^*(\mc{X}_\infty, \mb{Q})$.
This is a slightly stronger notion. In particular, it is possible that wedge product of two elements from odd 
degree cohomology groups become a Hodge class, although they are individually not Hodge classes. This 
is a common phenomena appearing in the cohomology of abelian varieties, for example. This will play a crucial role 
below to produce new examples of Mumford-Tate families. 

In order to study the tensor algebras more effectively, we separate  the odd 
cohomology groups  from the even ones. We take exterior algebra of the odd cohomology groups and 
the symmetric algebra of the even ones. This is done to preserve compatibility with cup-products.
Given two $r$-tuple of positive integers $\un{m}:=(m_1,...,m_r)$ and 
 $\un{k}:=(k_1,...,k_r)$, denote by 
 \begin{align*}
 & \mb{T}_{\un{m}}^{\un{k}}:= \bigwedge^{k_1} H^{m_1}(\mc{X}_{\infty}, \mb{Q}) \otimes ... \otimes \bigwedge^{k_r} H^{m_r}(\mc{X}_{\infty}, \mb{Q}),\, \mbox{ if each } m_i 
 \mbox{ is odd}, \\
 & \mb{T}_{\un{m}}^{\un{k}}:= \mr{Sym}^{k_1} H^{m_1}(\mc{X}_{\infty}, \mb{Q}) \otimes ... \otimes \mr{Sym}^{k_r} H^{m_r}(\mc{X}_{\infty}, \mb{Q}),\, \mbox{ if each } m_i 
 \mbox{ is even}.
 \end{align*}
 Given an $r$-tuple of even positive integers $\un{m}:=(m_1,...,m_r)$, an $l$-tuple of odd positive integers $\un{n}:=(n_1,...,n_l)$ and an 
 $r$ (resp. $l$) tuple of arbitrary positive integers $\un{k}:=(k_1,...,k_r)$ (resp. $\un{k}':=(k'_1,...,k'_l)$), denote by 
 \[\mb{T}^{(\un{k},\un{k}')}_{(\un{m},\un{n})} \mbox{ the pure part of } \mb{T}_{\un{m}}^{\un{k}} \otimes \mb{T}^{\un{k}'}_{\un{n}} \mbox{ i.e., }
 \mb{T}^{(\un{k},\un{k}')}_{(\un{m},\un{n})}:= \mr{Gr}^W_a \mb{T}_{\un{m}}^{\un{k}} \otimes \mb{T}^{\un{k}'}_{\un{n}},\]
 where $a:= \sum_{i=1}^r  m_ik_i+\sum_{j=1}^l n_jk'_j$. Denote by 
 \begin{equation}\label{eq:ten}
  \mb{T}_{(\un{m},\un{n})}:= \bigoplus_{(\un{k},\un{k}')} \mb{T}_{(\un{m},\un{n})}^{(\un{k},\un{k}')},
 \end{equation}
 where $\un{k}$ and $\un{k}'$ ranges over all $k$-tuple and $l$-tuple of positive integers, respectively.
Denote by
\[\mb{T}^s_{(\un{m},\un{n})} \mbox{ the same as } \mb{T}_{(\un{m},\un{n})} \mbox{ with } \mc{X}_{\infty} 
 \mbox{ replaced by } \mc{X}_s \mbox{ for any } s \in \Delta^*.\]
 Note that,  the Hodge structure on $H^m(\mc{X}_s,\mb{Q})$ is pure for all $m$, so the ``pure part'' condition is 
 redundant in this case.
Let $\mr{MT}^s_m$ be the Mumford-Tate group associated to the pure Hodge structure $H^m(\mc{X}_s,\mb{Q})$.
Then, the product of the Mumford-Tate groups
\[\mr{MT}_{(\un{m},\un{n})}^s:= \mr{MT}_{m_1}^s \times \mr{MT}_{m_2}^s \times ... \times \mr{MT}_{m_r}^s \times 
  \mr{MT}_{n_1}^s \times \mr{MT}_{n_2}^s \times ... \times \mr{MT}_{n_l}^s \]
acts on $\mb{T}_{(\un{m},\un{n})}^s$.
The family $\pi$ is called \emph{strictly Mumford-Tate with respect to} $(\un{m},\un{n})$ if for any Hodge class 
$\gamma \in \mb{T}_{(\un{m},\un{n})}$ and $s \in \mf{h}$ general, $j_s^*(\gamma)$ is fixed by $\mr{MT}_{(\un{m},\un{n})}^s$, where 
\[j_s^*: \mb{T}_{(\un{m},\un{n})} \to \mb{T}_{(\un{m},\un{n})}^s\]
is induced by the pullback of the natural inclusion of $\mc{X}_s$ inside $\mc{X}_{\infty}$.

\begin{prop}\label{prop:cur}
 Let $\pi_1: \mc{X} \to \Delta$ be a flat, projective family of genus $g$ curves for $g \ge 2$. 
 We assume that $\pi_1$ is smooth over $\Delta^*$ and the central 
 fiber is a very general irreducible nodal curve (in the sense of \cite{hodc}). 
 Then, $\pi_1$ is strictly Mumford-Tate with respect to $((0,2),(1))$.
\end{prop}

\begin{proof}
 Consider the family of Jacobians associated to the family of curves $\pi_1$,
 \[\pi_2: \mbf{J} \to \Delta^* \mbox{ i.e., for all } t \in \Delta^*,\, \pi_2^{-1}(t)=\mr{Jac}(\mc{X}_t).\]
 By the definition of cohomology of abelian varieties, there is a natural isomorphism of mixed Hodge structures
 between $H^1(\mc{X}_{\infty},\mb{Q})$ and $H^1(\mbf{J}_{\infty}, \mb{Q})$. This induces an isomorphism of 
 mixed Hodge structures,
 \[\bigwedge^* H^1(\mc{X}_{\infty},\mb{Q}) \xrightarrow{\sim} H^*(\mbf{J}_{\infty}, \mb{Q}).\]
 By \cite[Theorem $4.3$]{hodc}, we have 
 \[H^*_{\mr{Hdg}}(\mbf{J}_{\infty}, \mb{Q}) \cong \mb{Q}[\theta]/(\theta^{g+1}),\, \mbox{ where } g=\mr{genus}(\mc{X}_t),\, t \in \Delta^*.\]
 Note that, $\mr{Sym}^* H^0(\mc{X}_\infty,\mb{Q}) \cong \mb{Q}[T_0]$ and  $\mr{Sym}^* H^2(\mc{X}_\infty,\mb{Q}) \cong \mb{Q}[T_1]$
 where $T_0$ and $T_1$ are Hodge classes. 
 Consider the direct sum of vector spaces  $\mb{T}_{(0,2),(1)}$ as in \eqref{eq:ten}  associated to the family $\pi_1$. 
 Then, the space of Hodge classes $\mb{T}_{\mr{Hdg}}$ in $\mb{T}_{(0,2),(1)}$ is 
 isomorphic to $\mb{Q}[T_0,T_1,\theta]/(\theta^{g+1})$. Similarly, the set of Hodge class
 $\mb{T}^s_{\mr{Hdg}}$ in $\mb{T}^s_{(0,2),(1)}$
 contains  $\mb{Q}[T_0^s,T_1^s,\theta^s]/((\theta^s)^{g+1})$, where $(-)^s:=j_s^*(-)$. Hence, $T_0^{s}, T_1^s$ and 
 $\theta^s$ are fixed by the Mumford-Tate group $\mr{MT}_{(0,2),(1)}^s$. Therefore, $\pi_1$ is strictly
 Mumford-Tate with respect to $((0,2),(1))$. This proves the proposition.
\end{proof}

 \subsection{Cohomologies generated by Chern classes}\label{sec:cohgen}
 
 Let $X$, $Y$ be smooth, projective varieties of dimension $m$ and $n$, respectively.  
 Combining K\"{u}nneth decomposition with Poincare duality, we have for every $i, k \ge 0$,
 \begin{equation}\label{eq:kun}
  H^{2i-k}(X \times Y) \simeq \bigoplus_k H^{2n-k}(X)\otimes {H^{2i-k}(Y)}^{\vee} \simeq \bigoplus_{k} \Hom(H^{2m-k}(X), H^{2i-k}(Y)).
 \end{equation}
 Let $\mc{E}$ be a coherent sheaf on the fibre product $X \times Y$ and $c_i(\mc{E})$ be the $i$-th Chern class of $\mc{E}$. Denote by
 $\Phi_{\mc{E}}^{(i,k)}$
  the projection of $c_i(\mc{E})$ in $H^{2i -k}(Y)$ to the component $\Hom(H^{2m-k}(X), H^{2i-k}(Y))$. By \cite[Lemma $11.41$]{v4}, the induced morphism 
  \begin{equation}\label{eq:corr1}
    \Phi_{\mc{E}}^{(i,k)}: H^{2m-k}(X) \to H^{2i-k}(Y)\, \mbox{ is a morphism of pure Hodge structures}.
  \end{equation}

 \begin{thm}\label{thm:mix}
  Let $\pi_1: \mc{X}^{*} \to \Delta^{*}$ and $\pi_2: \mc{Y}^{*}  \to \Delta^{*}$ be two smooth, projective families of relative 
  dimensions $m$ and $n$, respectively. Assume that there exists a coherent sheaf $\mc{U}$ over $\mc{X}^{*}\times_{\Delta^{*}} \mc{Y}^{*}$ such that it is flat over $\Delta^{*}$.  Then the morphism 
  \[ \Phi_{\mc{U}_t}^{(i,k)} : H^{2m-k}(\mc{X}_t) \to H^{2i-k}(\mc{Y}_t) \]
induces a morphism of (limit) mixed Hodge structures: 
\[\Phi_{\mc{U}, \infty}^{(i,k)}: H^{2m-k}(\mc{X}_{\infty}) \to H^{2i-k}(\mc{Y}_{\infty}) .\] Furthermore, 
the morphisms $\Phi_{\mc{U}, \infty}^{(i,k)}$ and $\Phi_{\mc{U}_t}^{(i,k)}$ commute with pullback to 
closed fibers i.e., 
for any $u \in \mathfrak{h}$ with $e(u) = t$ (where $e$ is the exponential map) 
we have the following commutative diagram:
\begin{equation}\label{eq:comm1}
 \begin{diagram}
    H^{2m-k}(\mc{X}_{\infty})&\rTo^{\Phi_{\mc{U}, \infty}^{(i,k)}}&H^{2i-k}(\mc{Y}_{\infty})\\
    \dTo^{sim}^{(j_{u})^{*}}_{\cong}&\circlearrowleft&\dTo^{(j'_{u})^{*}}_{\cong}\\
    H^{2m-k}(\mc{X}_t)&\rTo^{\Phi_{\mc{U}_t}^{(i,k)}}&H^{2i-k}(\mc{Y}_t)
  \end{diagram}
\end{equation}
  where $j_u: \mc{Y}_t \hookrightarrow \mc{Y}_{\infty}$ and $j_u': \mc{X}_t \hookrightarrow \mc{X}_{\infty}$ are natural inclusions.
  \end{thm}

 \begin{proof}
 Consider the natural projective morphisms:
 \[\pi: \mc{X}^{*}\times_{\Delta^*}\mc{Y}^{*} \to \Delta^*,\, \pi_1: \mc{X}^* \to \Delta^*\, \mbox{ and } \pi_2: \mc{Y}^* \to \Delta^*. \]
  Consider the local system  $\mb{H}^{2i}:=R^{2i}{\pi}_{*}\mb{Z}$ over $\Delta^*$.
  We denote by \[\mb{H}^{i}_{\mc{X}^*}:=R^i{\pi_1}_{*}\mb{Z}\, \mbox{ and }\,
  \mb{H}^{i}_{\mc{Y}^*}:=R^i{\pi_2}_{*}\mb{Z}.\]
  By Künneth decomposition in families (see \cite[Ex. II.$18$]{kashi}), we have 
  \[\mb{H}^{2i} \simeq \bigoplus_{k}(\mb{H}^{k}_{\mc{X}^*} \otimes \mb{H}^{2i-k}_{\mc{Y}^*}) \]
  Applying Poincare duality to the local system $\mb{H}^{k}_{\mc{X}^*}$ (see \cite[\S I.$2.6$]{kuli}), we get:
  \[\mb{H}^{2i} \simeq \bigoplus_{k}(\mb{H}^{2m-k}_{\mc{X}^*})^{\vee} \otimes \mb{H}^{2i-k}_{\mc{Y}^*} \simeq 
  \bigoplus_k \mr{Hom}(\mb{H}^{2m-k}_{\mc{X}^*},\mb{H}^{2i-k}_{\mc{Y}^*}). \]
  For any $i$, the $i$-th Chern class $c_i(\mc{U})$ defines a global section of $\mb{H}^{2i}$. 
  Consider the projection $\phi$ of $c_i(\mc{U})$ to $\mr{Hom}(\mb{H}^{2m-k}_{\mc{X}^*},\mb{H}^{2i-k}_{\mc{Y}^*})$.
  Pulling back the morphism $\phi$ of local systems on $\Delta^{*}$ to the upper half plane $\mathfrak{h}$ 
    and taking global sections, 
    we get the morphism \[\Phi_{\mc{U}, \infty}^{(i,k)}: H^{2m-k}(\mc{X}_{\infty}) \to H^{2i-k}(\mc{Y}_{\infty}).\]
    Restricting the morphism to the fiber over $u \in \mf{h}$ gives us the morphism $\Phi_{\mc{U}_t}^{(i,k)}$, where $t:=e(u)$.
    In particular, we have commutative diagram \eqref{eq:comm1}.

  It remains to check that $\Phi_{\mc{U}, \infty}^{(i,k)}$ is a morphism of limit mixed Hodge structures.
  By \eqref{eq:corr1}, $\Phi_{\mc{U}_t}^{(i,k)}$ is a morphism of pure Hodge structures. Since the limit Hodge filtrations on 
  $\mc{X}_\infty$ and $\mc{Y}_\infty$ arise simply as a limit of these Hodge filtrations, we conclude that 
   $\Phi_{\mc{U}, \infty}^{(i,k)}$ preserves the limit Hodge filtrations. It remains to check that 
   $\Phi_{\mc{U}, \infty}^{(i,k)}$ preserves the limit weight filtration.
  Equivalently,
  using the diagram \eqref{eq:comm1}  we need to prove that $\Phi_{\mc{U}_t}^{(i,k)}$ preserves the weight filtration 
  where the weight filtration on $\mc{X}_t$ and $\mc{Y}_t$ 
  is induced by $\mc{X}_{\infty}$ and $\mc{Y}_{\infty}$, respectively (via the isomorphisms $j_u^*$ and ${j'_u}^*$, respectively).
  Recall, the weight filtration on $\mc{X}_t$ and $\mc{Y}_t$ is induced by the log of the monodromy operators 
  (see \cite[Lemma-Definition $11.9$]{pet}):
  \[N_{\mc{X}}:= \log(T_{\mc{X}})\, \mbox{ and }\, N_{\mc{Y}}:= \log(T_{\mc{Y}}).\]
  So, it suffices to check that for all $\gamma \in H^{2m-k}(\mc{X}_t)$,
  we have $\Phi_{\mc{U}_t}^{(i,k)}(N_{\mc{X}}(\gamma)) = N_{\mc{Y}}\Phi_{\mc{U}_t}^{(i,k)}(\gamma)$.
    Since $c_i(\mc{U})$ is a global section of the local system, it is monodromy invariant. This means the 
  induced morphism $\phi$ from 
  $\mb{H}^{2m-k}_{\mc{X}^*}$ to $\mb{H}^{2i-k}_{\mc{Y}^*}$ commutes with the monodromy operators i.e., 
  for every $t \in \Delta^*$, we have following commutative diagram:
  \begin{equation}\label{eq:comm2}
   \begin{diagram}
    H^{2m-k}(\mc{X}_{t})&\rTo^{\Phi_{\mc{U}_t}^{(i,k)}}&H^{2i-k} (\mc{Y}_{t})\\
    \dTo^{T_{\mc{X}}}&\circlearrowleft&\dTo^{T_{\mc{Y}}}\\
    H^{2m-k}(\mc{X}_t)&\rTo^{\Phi_{\mc{U}_t}^{(i,k)}}&H^{2i-k}(\mc{Y}_t)
  \end{diagram}
  \end{equation}
  where $T_{\mc{X}}$ and $T_{\mc{Y}}$ are the monodromy operators and $\Phi_{\mc{U}_t}^{(i,k)}$ is as in \eqref{eq:corr1}.
  This implies   for all $\gamma \in H^{2m-k}(\mc{X}_t)$, we have 
 $\Phi_{\mc{U}_t}^{(i,k)}(T_{\mc{X}}(\gamma)) = T_{\mc{Y}}\Phi_{\mc{U}_t}^{(i,k)}(\gamma)$.
  Hence, 
  \[  \Phi_{\mc{U}_t}^{(i,k)}(T_{\mc{X}}-\mr{Id})(\gamma) = \Phi_{\mc{U}_t}^{(i,k)}(T_{\mc{X}}(\gamma))-\Phi_{\mc{U}_t}^{(i,k)}(\gamma) = 
  T_{\mc{Y}}(\Phi_{\mc{U}_t}^{(i,k)}(\gamma))-\Phi_{\mc{U}_t}^{(i,k)}(\gamma)=(T_{\mc{Y}}-\mr{id}) \Phi_{\mc{U}_t}^{(i,k)}(\gamma).\]
  More generally, this implies for all $m \ge 1$,
  \[ \Phi_{\mc{U}_t}^{(i,k)}(T_{\mc{X}}-\mr{Id})^{m}(\gamma) = \Phi_{\mc{U}_t}^{(i,k)}(T_{\mc{X}}-\mr{Id})(T_{\mc{X}}-\mr{Id})^{m-1}(\gamma) = (T_{\mc{Y}}-\mr{Id})\Phi_{\mc{U}_t}^{(i,k)}(T_{\mc{X}}-\mr{Id})^{m-1}(\gamma)\]
  Therefore, by recursion we have  $\Phi_{\mc{U}_t}^{(i,k)}(T_{\mc{X}}-\mr{Id})^{m}(\gamma)= (T_{\mc{Y}}-\mr{Id})^{m}\Phi_{\mc{U}_t}^{(i,k)}(\gamma)$. Using the logarithmic expansion of $N_{\mc{X}}$ and $N_{\mc{Y}}$ we conclude:
  \[\Phi_{\mc{U}_t}^{(i,k)}(N_{\mc{X}}(\gamma)) = N_{\mc{Y}}\Phi_{\mc{U}_t}^{(i,k)}(\gamma),\, \mbox{ for all } 
  \gamma \in H^{2m-k}(\mc{X}_t).  \]
 This implies that $\Phi_{\mc{U}_t}^{(i,k)}$ preserves the limit weight filtration. This proves the theorem.
  \end{proof}

 \begin{defi}\label{defi:coh}
  Let $X$,$Y$ be smooth, projective varieties of dimensions $m$ and $n$, respectively. 
  Denote by $\mc{E}$ a coherent sheaf on $X \times_k Y$. 
  The variety $Y$ is said to be \emph{cohomologically generated by} $(X, \mc{E})$ if there is a collection 
    $S_{Y}(X,\mc{E})$ 
  of pairs of integers $(k,i)$
  such that $H^{*}(Y)$ is generated as a cohomology ring by
   the direct sum of the images of \[\Phi_{\mc{E}}^{(i,k)}: H^{2m-k}(X) \to H^{2i-k}(Y)\]
   as the pair $(k,i)$ varies over all the elements in $S_Y(X,\mc{E})$.
   Note that $\pr_1(S_{Y}(X,\mc{E}))$ need not contain all integers from $0$ to $2m$.
  We call $S_Y(X,\mc{E})$ an \emph{associated indexing set}.
  \end{defi}

  \begin{notac}\label{defi:pair}
  We fix the following notations:
  \begin{align*}
   & S_{\mr{even}} := \{(k,i) \in S_{Y}(X, \mc{E}) \ | \  k  \ \mr{even} \}\, \mbox{ and }\, 
     S_{\mr{odd}} := \{(k,i) \in S_{Y}(X, \mc{E}) \ | \  k  \ \mr{odd} \}\\
  & p(S_{\mr{even}}):= \{2m-k | (k,i) \in S_{\mr{even}}\}\, \mbox{ and }\, p(S_{\mr{odd}}):= \{2m-k | (k,i) \in S_{\mr{odd}}\}\\
  & q(S_{\mr{even}}):= \{2i-k | (k,i) \in S_{\mr{even}}\}\, \mbox{ and }\, q(S_{\mr{odd}}):= \{2i-k | (k,i) \in S_{\mr{odd}}\}
  \end{align*}
 \end{notac}

 \begin{thm}\label{thm:mt}
    Let $\pi_1: \mc{X}^{*} \to \Delta^{*}$ and $\pi_2: \mc{Y}^{*}  \to \Delta^{*}$ be two smooth, projective families
    of relative dimensions $m$ and $n$, respectively. Assume that there exists a 
    coherent sheaf $\mc{U}$ over $\mc{X}^{*}\times_{\Delta^{*}} \mc{Y}^{*}$ such that it is flat over $\Delta^{*}$
    and for general $t \in \Delta^*$,
    $\mc{Y}_t$ is cohomologically generated by $(\mc{X}_t, \mc{U}_t)$ by an indexing set $S_{\mc{Y}_t}(\mc{X}_t,\mc{U}_t)$ such that 
     $\pi_1$
    is strictly Mumford-Tate with respect to $(p(S_{\mr{even}}),p(S_{\mr{odd}}))$. Then, the family $\pi_2$ is 
    Mumford-Tate. 
 \end{thm}

 \begin{proof}
 Let $t \in \Delta^*$ be such that $\mc{Y}_t$ is cohomologically generated by $(\mc{X}_t, \mc{U}_t)$ with indexing set 
 $S_{\mc{Y}_t}(\mc{X}_t,\mc{U}_t)$ such that $\pi_1$
    is strictly Mumford-Tate with respect to $(p(S_{\mr{even}}),p(S_{\mr{odd}}))$.
  Using Ehresmann's theorem one can check that for any $s \in \Delta^*$, $\mc{Y}_s$ is cohomologically generated by 
  $(\mc{X}_s, \mc{U}_s)$ and we have an equality of indexing sets 
  $S_{\mc{Y}_t}(\mc{X}_t,\mc{U}_t)=S_{\mc{Y}_s}(\mc{X}_s,\mc{U}_s)$. 
 Denote by 
 \[\mb{T}_{\mc{X}}:= \mb{T}_{(p(S_{\mr{even}}),p(S_{\mr{odd}}))} \mbox{ and }\mb{T}_{\mc{Y}}:= \mb{T}_{(q(S_{\mr{even}}),q(S_{\mr{odd}}))}\, \mbox{ with }
  \mc{X}_\infty\, \mbox{ replaced by } \mc{Y}_{\infty}. \]
Recall, for any $(k,i) \in S_{\mc{Y}_t}(\mc{X}_t,\mc{U}_t)$ we have the morphism $\Phi_{\mc{U},\infty}^{(i,k)}$ of mixed Hodge structures
from $H^{2m-k}(\mc{X}_\infty)$ to $H^{2i-k}(\mc{Y}_\infty)$. This 
  induces a morphism of mixed Hodge structures:
  \[\phi: \mb{T}_{\mc{X}} \to \mb{T}_{\mc{Y}}.\]
  Recall, the cup-product morphism is a morphism of mixed Hodge structures \cite[Lemma $6.16$]{fuji}.
  So, the composition of the cup-product morphism with $\phi$:
  \[\Phi: \mb{T}_{\mc{X}} \xrightarrow{\phi} \mb{T}_{\mc{Y}} \xrightarrow{\bigcup} H^*(\mc{Y}_\infty, \mb{Q})\]
  is a morphism of mixed Hodge structures. Given $s \in \Delta^*$, denote by (see \S \ref{sec:str})
   \[\mb{T}_{\mc{X}_s}:= \mb{T}^s_{(p(S_{\mr{even}}),p(S_{\mr{odd}}))}\, \mbox{ and }  \mb{T}_{\mc{Y}_s}:= \mb{T}^s_{(q(S_{\mr{even}}),q(S_{\mr{odd}}))}\, \mbox{ with } \mc{X}_s \mbox{ replaced by } \mc{Y}_s. \]
   As before, we have the following composed morphism of Hodge structures:
   \[\Phi_s: \mb{T}_{\mc{X}_s} \to \mb{T}_{\mc{Y}_s} \xrightarrow{\bigcup} H^*(\mc{Y}_s, \mb{Q}),\]
where the first morphism arises from $\Phi_{\mc{U}_s}^{(i,k)}$ as $(k,i)$ ranges over entries in 
$S_{\mc{Y}_s}(\mc{X}_s,\mc{U}_s)$.
  By Theorem \ref{thm:mix} we then have the following commutative diagram:
  \[\begin{diagram}
     \mb{T}_{\mc{X}}&\rTo^{\Phi}&H^*(\mc{Y}_\infty, \mb{Q})\\
     \dTo^{j_s^*}&\circlearrowleft&\dTo_{(j'_s)^*}\\
     \mb{T}_{\mc{X}_s}&\rTo^{\Phi_s}&H^*(\mc{Y}_s, \mb{Q})    
    \end{diagram}\]
where $j_s$ (resp. $j'_s$) is the natural inclusion of $\mc{X}_s$ (resp. 
$\mc{Y}_s$) into $\mc{X}_\infty$ (resp. $\mc{Y}_\infty$). 

Take $\gamma \in F^pH^{2p}(\mc{Y}_\infty, \mb{Q})$ i.e., $\gamma$ is a Hodge class.
We need to prove that ${j'_s}^*(\gamma)$ is a Hodge class in $H^{2p}(\mc{Y}_s,\mb{Q})$.
Since $\mc{Y}_s$ is cohomologically generated by $(\mc{X}_s,\mc{U}_s)$ and $\Phi$ is a morphism of mixed Hodge structures, 
there exists a Hodge class 
$\gamma' \in \mb{T}_{\mc{X}}$ such that $\Phi(\gamma')=\gamma$.
As $\pi_1$ is strictly Mumford-Tate with respect to $(p(S_{\mr{even}}),p(S_{\mr{odd}}))$, we have $j_s^*(\gamma')$ is fixed by 
$\mr{MT}_{(p(S_{\mr{even}}),p(S_{\mr{odd}}))}^s$. Hence, $j_s^*(\gamma')$ is a Hodge class in $\mb{T}_{\mc{X}_s}$. Since $\Phi_s$
is a morphism of Hodge structures, this means 
\[(j_s')^*(\gamma)=\Phi_s \circ j_s^*(\gamma') \mbox{ is a Hodge class}.\]
Therefore, $\pi_2$ is a Mumford-Tate family. This proves the theorem.
 \end{proof}
 
 We now use the above theorem to get an explicit example.
 
 \begin{cor}\label{cor:exa}
Let $\pi_1: \mc{X} \to \Delta$ be a flat, projective family of curves satisfying the hypothesis in Proposition \ref{prop:cur}.
Fix an invertible sheaf $\mc{L}$ on $\mc{X}^*:= \pi_1^{-1}(\Delta^*)$ of (relative) odd degree over the punctured disc $\Delta^*$. 
Let \[\pi_2: \mc{M}(2,\mc{L}) \to \Delta^*\]
be a relative moduli space of rank $2$ semi-stable sheaves with fixed determinant $\mc{L}$ over $\mc{X}^*$.
Then, $\pi_2$ is a Mumford-Tate family.
 \end{cor}
 
 \begin{proof}
  Consider the universal bundle $\mc{U}$ over $\mc{X}^* \times_{\Delta^*} \mc{M}(2,\mc{L})$. It is well-known 
  that for each $t \in \Delta^*$, the fiber $\mc{M}(2,\mc{L})_t:= \pi_2^{-1}(t)$ is cohomologically generated by 
  $(\mc{X}_t,\mc{U}_t)$ with the associated indexing set (see \cite[Theorem $1$]{new1}):
  \[ \{(0,1),(0,2),(1,2),(2,2)\}\]
  By Proposition \ref{prop:cur}, $\pi_1$ is strictly Mumford-Tate. Then, Theorem \ref{thm:mt} implies that $\pi_2$
  is a Mumford-Tate family. This proves the corollary.
 \end{proof}

 \begin{rem}
  In fact, the relative moduli space $\mc{M}(2,\mc{L})$ mentioned in Corollary \ref{cor:exa}
  degenerates to a singular variety. A desingularization
  of this variety satisfies the classical Hodge conjecture. See \cite[Theorem $5.2$]{hodc} for details.
 \end{rem}

 \section*{Acknowledgements} 
This article was motivated by some questions asked by
Prof. C. Simpson, after the second author gave a talk on the article \cite{hodc} at the workshop `Moduli of bundles and related structures' held at ICTS, Bengaluru, India. We thank Prof. Simpson for his interest and the organisers for organising the workshop. 
We also thank Prof. R. Laterveer for his comments on an earlier draft.

\end{document}